\documentclass{article} 

\usepackage{latexsym}
\usepackage{amsfonts}
\usepackage{amssymb}
\usepackage{amsmath}
\usepackage{amscd}
\usepackage{eucal}
\usepackage[all, knot]{xy}
\xyoption{arc}
\xyoption{rotate}
\xyoption{dvips}
\usepackage{amsthm}
\usepackage[dvips]{pict2e}
\usepackage{pstricks}

\usepackage{verbatim} 
\usepackage[dvips]{graphicx}

\numberwithin{equation}{section}

\newtheorem{theorem}{Theorem}[section]

\newtheorem{proposition}[theorem]{Proposition}

\newtheorem{lemma}[theorem]{Lemma}
\newtheorem{definition}[theorem]{Definition}

\newtheorem{example}[theorem]{Example}
\newtheorem{examples}[theorem]{Examples}
\newtheorem{remarks}[theorem]{Remarks}

\newtheorem{thm}[theorem]{Theorem}

\newtheorem{prop}[theorem]{Proposition}
\newtheorem{cor}[theorem]{Corollary}

\newenvironment{mydefinition}{\begin{definition}\upshape} {\end{definition}}
\newenvironment{myremark}{\begin{remark}\upshape} {\end{remark}}
\newenvironment{myremarks}{\begin{remarks}\upshape} {\end{remarks}}
\newenvironment{myexample}{\begin{example}\upshape} {\end{example}}
\newenvironment{myexamples}{\begin{examples}\upshape} {\end{examples}}

\newtheoremstyle{example}{\topsep}{\topsep}%
     {}
     {}
     {\bfseries}
     {.}
     {8pt}
     {\thmname{#1}\thmnumber{ #2}\thmnote{ #3}}

   \theoremstyle{example}

\newtheorem{remark}[theorem]{Remark}

\newcommand{\glob}{
\xy
(-5,0)*+{.}="1";
(5,0)*+{.}="2";
{\ar@/^1pc/^{} "1";"2"};
{\ar@/_1pc/_{} "1";"2"};
{\ar@{=>}^{} (0,2)*{};(0,-2)*{}} ;
\endxy}

\newcommand{\iso}{\cong}
\newcommand{\catequiv}{\simeq}
\newcommand{\vardelta}{\partial}
\newcommand{\lra}{\ensuremath{\longrightarrow}}
\newcommand{\cat}[1]{\ensuremath{\mbox{\bfseries {\upshape {#1}}}}}
\newcommand{\cl}[1]{\ensuremath{\mathcal {#1}}}
\newcommand{\bb}[1]{\ensuremath{\mathbb {#1}}}
\newcommand{\ed}{\end{document}}
\newcommand{\map}[1]{\ensuremath{\stackrel{{#1}}{\lra}}}

\newcommand{\demph}[1]{{\bfseries #1}}

\newcommand{\ncat}[1]{\cat{$#1$-Cat}}
\newcommand{\bncat}[1]{\cat{$(#1)$-Cat}}

\newcommand{\numarabic}{\renewcommand{\labelenumi}{\arabic{enumi}.}}

\newenvironment{prf}{\vspace{2ex}\begin{sloppypar}{\noindent\upshape
{\bfseries Proof. }}} {{\hspace*{\fill}
$\Box$}\end{sloppypar}\vspace{2ex}}

\newcommand{\psinv}[6]{\xy
(#1,0)*+{#3}="x";
(#2,0)*+{#5}="y";
{\ar@<.7ex>^{#4} "x"; "y"};
{\ar@<.7ex>^{#6} "y"; "x"};
\endxy
}

\newcommand{\op}{\ensuremath{\mbox{\hspace{1pt}{\scriptsize\upshape op}}}}

\newcommand{\fc}{\cat{fc}}

\newcommand{\dcorner}{
\xy
(-2,2)*{}; (0,0) **\dir{-};
(2,2)*{}; (0,0) **\dir{-};
\endxy}

\newcommand{\ucorner}{
\xy
(-2,-2)*{}; (0,0) **\dir{-};
(2,-2)*{}; (0,0) **\dir{-};
\endxy}

\newcommand{\citencat}{\cite{str2,bd1,bat1,pen1,tam1,joy1,lei5,lei7,hmp3,hmp4,tri1}}

\begin{document}

\title{Comparing operadic theories of $n$-category}

\author{Eugenia Cheng\\Department of Mathematics, Universit\'e de Nice Sophia-Antipolis \\ and \\ Department of Mathematics, University of Sheffield \\E-mail: e.cheng@sheffield.ac.uk}

\maketitle

\begin{abstract}

We give a framework for comparing on the one hand theories of $n$-categories that are weakly enriched operadically, and on the other hand $n$-categories given as algebras for a contractible globular operad. Examples of the former are the definition by Trimble and variants (Cheng-Gurski) and examples of the latter are the definition by Batanin and variants (Leinster).  We first provide a generalisation of Trimble's original theory that allows for the use of other parametrising operads in a very general way, via the notion of categories weakly enriched in \cl{V} where the weakness is parametrised by a \cl{V}-operad $P$. We define weak $n$-categories by iterated weak enrichment using a series of parametrising operads $P_i$.  We then show how to construct from such a theory an $n$-dimensional globular operad for each $n\geq 0$ whose algebras are precisely the weak $n$-categories, and we show that the resulting globular operad is contractible precisely when the operads $P_i$ are contractible.  We then show how the globular operad associated with Trimble's topological definition is related to the globular operad used by Batanin to define fundamental $n$-groupoids of spaces.
\end{abstract}

\newpage

\setcounter{tocdepth}{2}
\tableofcontents

\section*{Introduction}

\addcontentsline{toc}{section}{Introduction}

Many different notions of weak $n$-category have been proposed in the literature \citencat, and one of the most fundamental open questions in the subject concerns the relationships between these theories.  Few comparison functors have been constructed, let alone full equivalences between theories, although various relationships are widely suspected.

The aim of this paper is to compare the theories of Trimble \cite{tri1} and Batanin \cite{bat1}.  The consequences of this comparison go beyond a mere technicality of the foundations of higher-dimensional category theory.  Trimble's theory is the only one that explicitly uses classical operads, so this comparison opens up the possibility of using the huge and well-developed theory of classical operads, including all the topological and homotopy theoretic techniques developed for that subject, for the study of $n$-categories.

 In this case, unlike when comparing other theories (see for example \cite{che7,che8}), it is not the underlying shapes of cells that is the main issue, but rather the way in which composition and coherence are handled.  Each of these definitions uses the formalism of operads to control the operations of an $n$-category, but in very different ways.  Trimble's definition uses a classical (non-symmetric) operad iteratively, whereas Batanin's definition uses a globular operad non-iteratively.  The idea is that the operad operations of a given arity will be the different ways of composing a configuration of cells of that arity.  Since a classical operad  only has arities which are integers $k \geq 0$, it can \emph{a priori} only parametrise composites of arities $k \geq 0$.  This is enough for 1-categories, and indeed for the homotopy monoids with which operads achieved much of their early success \cite{may2}.  That is, in a 1-category a ``configuration for composition'', or pasting diagram, is just a string of composable arrows
\[\xy
(0,0)*+{a_0}="1";
(15,0)*+{a_1}="2";
(28,0)*+{}="3";
(35,0)*+{\ldots}="9";
(42,0)*+{}="33";
(57,0)*+{a_{k-1}}="4";
(73,0)*+{a_{k}}="5";
{\ar^{f_1} "1";"2"};
{\ar^{f_2} "2";"3"};
{\ar^{} "33";"4"};
{\ar^{f_k} "4";"5"};
\endxy\]
so its arity is completely specified by one integer $k$.  However, in a higher-dimensional category we have composable diagrams such as
\vspace{-6pt}\[
       \def\objectstyle{\scriptstyle}
       
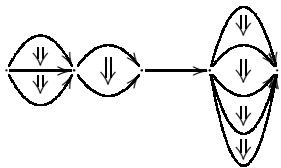
\xy
   (18,0)*{\cdot}="1";
   (27,0)*{\cdot}="2";
       {\ar@/^.8pc/ "1";"2"};
       {\ar@/_.8pc/ "1";"2"};
       {\ar@/^2pc/ "1";"2"};
       {\ar@/_2pc/ "1";"2"};
       {\ar@/_3pc/ "1";"2"};
       {\ar@{=>} (22.5,1.5)*{};(22.5,-1.5)*{}} ;
       {\ar@{=>} (22.5,7.25)*{};(22.5,4.75)*{}} ;
       {\ar@{=>} (22.5,-4.75)*{};(22.5,-7.25)*{}} ;
       {\ar@{=>} (22.5,-9.25)*{};(22.5,-11.75)*{}} ;
       (-9,0)*{\cdot}="1";
       (0,0)*{\cdot}="2";
           {\ar "1";"2"};
           {\ar@/^1.1pc/ "1";"2"};
           {\ar@/_1.1pc/ "1";"2"};
           {\ar@{=>} (-4.5,3)*{};(-4.5,.75)*{}} ;
           {\ar@{=>} (-4.5,-.75)*{};(-4.5,-3)*{}} ;
   (0,0)*{\cdot}="1";
   (9,0)*{\cdot}="2";
       {\ar@/^.8pc/ "1";"2"};
       {\ar@/_.8pc/ "1";"2"};
       {\ar@{=>} (4.5,1.75)*{};(4.5,-1.75)*{}} ;
   (9,0)*{\cdot}="1";
   (18,0)*{\cdot}="2";
    {\ar "1";"2"};
\endxy\]
which evidently cannot be completely specified by just one integer.  The reason is that we now have the possibility of many different types of composition -- along bounding 0-cells, 1-cells, 2-cells, and so on -- and a general $n$-dimensional pasting diagram may involve many of these different types of composition at once.

Trimble's approach uses an iterative process to build up these more complicated forms of composition.  Batanin, on the other hand, deals with this issue by starting with a more complicated form of operad -- a globular operad.  A globular operad has as arities not just integers $k \geq 0$, but \emph{globular pasting diagrams}.  These are exactly the arities we need for composition in an $n$-category.

The idea behind the comparison is to compare the operads in question.  Where Batanin's $n$-categories are controlled by a single ``contractible globular operad'', a ``Trimble-like'' theory of $n$-categories is controlled by a \emph{series} of classical operads -- each of which parametrises just one type of composition.  We will take such a series and construct from it a single globular operad that parametrises all types of composition simultaneously.  That is, it single-handedly encodes exactly the operations that the series of classical operads encoded collectively.  Our main theorem will be that, given any ``Trimble-like'' theory of $n$-categories, we have a contractible globular operad whose algebras are precisely the $n$-categories we started with. This is a generalisation of Leinster's Claim~10.1.9 in \cite{lei8}, in which he conjectures this result in the specific case of Trimble's original definition, whereas we have generalised it to a much broader framework.

First we must explain what we mean by ``Trimble-like'' theory of $n$-categories.  Trimble uses one specific operad to define his notion of $n$-category (extracting a series of operads from it), but it is possible and indeed desirable to generalise his framework to allow for the use of other suitable operads.  We are motivated by analogy with the study of loop spaces much of whose success is rooted in the use of different operads for different situations.  As a more specific example, we are motivated by the desire to be able to use a ``smooth version'' of Trimble's theory, to handle $n$-categories of cobordisms (see \cite{cg2}); Trimble's original theory is the ``continuous version''.

Thus we begin, in Section~\ref{onepointone} by stating Trimble's original definition, and continue, in Section~\ref{onepointtwo} by generalising it to what we call an ``iterated operadic theory of $n$-categories''.  Given a finite product category \cl{V} and an operad $P$ in \cl{V} we introduce the notion of a $(\cl{V}, P)$-category, which is a cross between a $P$-algebra and a \cl{V}-category, and is to be thought of as a ``category enriched in \cl{V} but weakened by the action of the operad $P$''.  The idea is that instead of having one composite for any given string of $k$ composable cells, we have one for each such string \emph{together with} an element of $P(k)$.  We can compare this with the generalisation from monoid to $P$-algebra.  A monoid can be expressed as a set $A$ together with, for each integer $k \geq 0$, a multiplication map:
\[A^k \lra A.\]
This gives, for every string of $k$ elements of $A$, an element of $A$ which is to be regarded as their product.  For a $P$-algebra, we now have a map:
\[P(k) \times A^k \lra A\]
 which gives, for every string of $k$ elements of $A$ \emph{together with} an element of $P(k)$, an element of $A$ which we might regard as a ``parametrised product'' of the $k$ elements of $A$.  For example, in the case of $A_\infty$-spaces \cite{sta1} $A$ and $P(k)$ are spaces; the multiplication parametrised by $P(k)$ is associative and unital up to homotopy as each $P(k)$ is required to be contractible.

We can apply this principle for composition in a category as well.  In an ordinary category $A$, composition is given by, for every integer $k \geq 0 $ and objects $a_0, \ldots, a_k$, a map:
\[A(a_{k-1}, a_k) \times \cdots \times A(a_0, a_1) \lra A(a_0, a_k).\]
This gives, for every string of $k$ composable morphisms of $A$, a morphism of $A$ which is to be regarded as their composite; it is sufficient (and usual) to specify these maps only for the cases $k = 0,2$ as the associativity axioms ensure that we then have a well-defined such map for each $k \geq 0$.  To parametrise this by an operad $P$, we now specify composition maps:
\[P(k) \times A(a_{k-1}, a_k) \times \cdots \times A(a_0, a_1) \lra A(a_0, a_k)\]
which give, for every string of $k$ composable morphisms of $A$ \emph{together with} an element of $P(k)$, a morphism of $A$ which we regard as a ``parametrised composite''.  As above, we can do this enriched in \cat{Top} or indeed in various other suitable enriching categories $\cl{V}$, and this gives a notion of ``weak enrichment'' that we can then iterate to form weak $n$-categories as follows.

Recall that standard enrichment can be iterated to produce higher categories, but we will only reach \emph{strict} $n$-categories in this way: we construct, for each $n \geq 0$ a category $\cl{S}_n$ of strict $n$-categories inductively as follows:

\[\begin{array}{rcl}
\cl{S}_0 &=& \cat{Set} \\
\cl{S}_1 &=& \cat{$\cl{S}_0$-Cat}\\
\cl{S}_2 &=& \cat{$\cl{S}_1$-Cat}\\
\cl{S}_3 &=& \cat{$\cl{S}_2$-Cat}\\
& \vdots  \\
\cl{S}_n &=& \cat{$\cl{S}_{n-1}$-Cat}\\
& \vdots  \\
\end{array}\]
For weak $n$-categories we need to use our weak form of enrichment, so for each $n \geq 0$ we define
\begin{itemize}
\item a category $\cl{V}_n$ of weak $n$-categories, and
\item an operad $P_n$ in $\cl{V}_n$ which will parametrise composition in weak $(n+1)$-categories.
\end{itemize}
We then construct weak $n$-categories by the following inductive process:
\[\begin{array}{rcl}
\cl{V}_0 &=& \cat{Set} \\
\cl{V}_1 &=& \cat{$(\cl{V}_0,P_0)$-Cat}\\
\cl{V}_2 &=& \cat{$(\cl{V}_1, P_1)$-Cat}\\
\cl{V}_3 &=& \cat{$(\cl{V}_2, P_2)$-Cat}\\
& \vdots  \\
\cl{V}_n &=& \cat{$(\cl{V}_{n-1}, P_{n-1})$-Cat}\\
& \vdots  \\
\end{array}\]
As a final proviso, note that we need to place a condition on the operads $P_i$ to ensure that the resulting $n$-categories are suitably coherent.  We introduce a notion of contractibility and demand that each of the operads $P_i$ is contractible.

In this framework it seems necessary to give a whole series of parametrising operads $P_i$ as part of the data when defining $n$-categories.  In fact Trimble's original definition starts with just one operad $E$ in topological spaces; one of the most elegant features of his definition is that  the series of operads $P_i$ is produced from the single operad $E$ as part of the inductive process.  However it is the series of operads $P_i$ that we need to make the comparison with Batanin's definition, so this is the framework we use for the rest of the work.  In other work \cite{cg2,et1} we focus on the iteration producing the $P_i$ from the single operad $E$; May's attempt to generalise Trimble's iteration appears in \cite{may1} but has been observed to be flawed by Batanin (the induction step does not go through).  Also note that as this definition is by induction, it only defines $n$-categories for finite $n$; in \cite{et1} we use a terminal coalgebra construction to construct an $\omega$-dimensional version of Trimble's definition.

It is useful to understand what each of these operads $P_i$ is parametrising. Given an integer $n \geq 0$, an $n$-category $A$ will be defined by
\begin{itemize}
\item a set $A_0$ of 0-cells, and
\item for every pair $a, a'$ of 0-cells, an $(n-1)$-category $A(a,a') \in \cl{V}_{n-1}$, equipped with
\item for every integer $k \geq 0$ and 0-cells $a_0, \ldots, a_k$, a composition morphism
    \[P_{n-1}(k) \times A(a_{k-1}, a_k) \times \cdots \times A(a_0, a_1) \lra A(a_0, a_k)\]
in $\cl{V}_{n-1}$ giving composition along bounding 0-cells
\end{itemize}
satisfying certain axioms.  Note that this composition morphism tells us how 0-composition of $(m+1)$-cells of $A$ is parametrised by the $m$-cells of $P_{n-1}$, for each $0 \leq m \leq n-1$.

Of course, for an $n$-category we also require composition along bounding $i$-cells for all $1 \leq i \leq n-1$, but here these are given inductively -- they are contained in the data for the hom-$(n-1)$-categories $A(a,a')$.  Thus we see that $i$-composition in $A$ is given by 0-composition in the hom-$(n-i)$-categories, so is parametrised by $P_{n-i-1}$.



All of this indexing is summed up in Table~\ref{bigtable}, with entries stating which dimension of which operad parametrises the given composition.

\begin{table}[htb]
\caption{\label{bigtable} Indexing for composition parametrised by operads $P_i$}
\vspace{2em}
\hspace*{-9em}
\begin{tabular}{|rc|ccccccc|}
\hline &&&&&&&&\\
&&&&& \makebox[0em]{\textsf{{\bfseries composition of}}} &&& \\[6pt]
& & \bfseries{1-cells}  &  \bfseries{2-cells}  &  \bfseries{3-cells}  &  \bfseries{4-cells}  & $\cdots$ & \bfseries{$(n-1)$-cells}  & \bfseries{$n$-cells}  \\[6pt]
\hline &&&&&&&& \\[-4pt]
&  \bfseries{0-cells}  & 0 of $P_{n-1}$ & 1 of $P_{n-1}$ & 2 of $P_{n-1}$ & 3 of $P_{n-1}$ & $\cdots$ & $(n-2)$ of $P_{n-1}$ & $(n-1)$ of $P_{n-1}$ \\[7pt]
&  \bfseries{1-cells}  & -- & 0 of $P_{n-2}$ & 1 of $P_{n-2}$ & 2 of $P_{n-2}$ & $\cdots$ & $(n-1)$ of $P_{n-2}$ & $(n-2)$ of $P_{n-2}$ \\[7pt]
&  \bfseries{2-cells}  & -- & -- & 0 of $P_{n-3}$ & 1 of $P_{n-3}$ & $\cdots$  & $(n-2)$ of $P_{n-3}$ & $(n-3)$ of $P_{n-3}$ \\[7pt]
 \textsf{\bfseries{along}} \hspace{-1em} &  \bfseries{3-cells}  & -- & -- & -- & 0 of $P_{n-4}$ & $\cdots$  & $(n-3)$ of $P_{n-4}$ & $(n-4)$ of $P_{n-4}$ \\[7pt]
& $\vdots$ & $\vdots$ & $\vdots$ & $\vdots$ & $\vdots$ & & $\vdots$ & $\vdots$ \\[6pt]
&  \bfseries{$(n-3)$-cells}  & -- & -- & -- & -- & $\cdots$  & 1 of $P_2$ & 2 of $P_2$\\[7pt]
&  \bfseries{$(n-2)$-cells}  & -- & -- & -- & -- & $\cdots$ & 0 of $P_1$ & 1 of $P_1$\\[7pt]
&  \bfseries{$(n-1)$-cells}  & -- & -- & -- & -- & $\cdots$ & -- & 0 of $P_0$\\[6pt]
\hline
\end{tabular}
\vspace{2em}
\end{table}

This pattern of shifting dimensions is what we need to encode abstractly when we compile the operads $P_i$ into one globular operad in Section~\ref{four}.  As a preliminary to this process, we show that each $(\cl{V},P)$-category construction is itself operadic.  This is the subject of Section~\ref{two}.  We recall the notion of $(\cl{E}, T)$-operad, where \cl{E} is a cartesian category and $T$ a cartesian monad on \cl{E}.  We then show how to start with a classical operad $P$ in \cl{V} as above, and construct an $(\cl{E}, T)$-operad $\Sigma P$ whose algebras are precisely the $(\cl{V}, P)$-categories of the previous section.  We will take \cl{E} to be the category \cat{\cl{V}-Gph} of graphs enriched in \cl{V}, and $T$ to be the free \cl{V}-category monad $\cat{fc}_{\cl{V}}$; these constructions are possible provided \cl{V} is suitably well-behaved.  $\Sigma P$ is a sort of ``suspension'' of $P$, so that $P$ can be thought of as the ``one-object'' version of $\Sigma P$, or rather, algebras for $P$ are one-object versions of algebras for $\Sigma P$ just as monoids are the one-object version of categories.

An immediate consequence is that we have a cartesian monad  $\cat{fc}_{(\cl{V},P)}$ on \cat{$\cl{V}$-Gph} giving free $(\cl{V},P)$-categories, and we will rely heavily on this in Section~\ref{four}.  This monadicity result could be proved directly but we have included the above results as we consider them to be interesting in their own right.

In Section 3 we briefly recall Batanin's definition of weak $n$-category.  In fact we use a non-algebraic version of Leinster's variant of this definition. Batanin's original definition uses \emph{contractible globular operads with a system of compositions}.  The idea is that the ``system of compositions'' ensures that binary and nullary composites exist, the globular operad keeps track of all the operations derived from the binary composites, and the contractibility ensures that the result is sufficiently coherent.  These special globular operads ``detect'' weak $n$-categories in the way that $A_\infty$-operads ``detect'' loop spaces -- a globular set is a weak $n$-category if and only if it is an algebra for any such operad.

Leinster's variant combines the notions of system of compositions and contraction into one more general notion of contraction.  This means that contractions now provide composites \emph{and} coherence cells.  The composition is ``unbiased'', that is, all arities of composition are specified, not just binary and nullary ones; as before, the globular operad keeps track of all the operations derived from these composites. Leinster then considers \emph{globular operads with contraction}, and uses these to ``detect'' weak $n$-categories as above.


In the present work we use the non-algebraic version of this definition, in that we consider \emph{contractible globular operads} rather than operads with a \emph{specified} contraction.  This means in effect that instead of specifying one composite for each arity of composition, we just ensure that such a composite \emph{exists} for each arity. Likewise, instead of specifying one coherence cell to mediate between any given pair of different composites of the same diagram, we simply demand that such a cell exists; if there are many, we do not insist on distinguishing one.  This is the version used by Berger and Cisinski in their work on Batanin's theory \cite{ber2,cis1}.

This makes for a more natural comparison with Trimble's theory, as we have already seen that, given a string of $k$ composable morphisms, we do not have a unique composite, but rather one composite for every element of $P(k)$.  Note that this does not mean Trimble's theory is not algebraic -- indeed the main result of this work shows that Trimble-like $n$-categories are algebras for a certain operad.  The point is that these $n$-categories can be considered to be algebraic once the series of parametrising operads has been fixed.

In Section 4 we give the main comparison construction.  We show that any iterative operadic theory of $n$-categories is operadic in the sense of Batanin/Leinster.  More precisely, consider an iterative operadic theory of $n$-categories given by a series of operads $P_i$ in categories $\cl{V}_i$, for $i \geq 0$, with
\begin{itemize}
\item $\cl{V}_0 = \cat{Set}$, and
\item for all $i \geq 0$, $\cl{V}_{i+1} = \cat{$(\cl{V}_i, P_i)$-Cat}$.
\end{itemize}
The main result of Section~\ref{four} is that, given this data, there is for each $n \geq 0$ a globular operad $Q^{(n)}$ such that the category of algebras for $Q^{(n)}$ is $\cl{V}_n$, the category of $n$-categories we started with.  Furthermore, $Q^{(n)}$ is contractible (in the sense of Leinster) if for all $0 \leq i \leq n-1$, the operad $P_i$ is contractible (in the sense of Section~\ref{onepointtwo}).

The idea behind the construction is to compile the classical operads $P_i$ into one globular operad.  A globular operad must have, for every pasting diagram $\alpha$, a set of operations of arity $\alpha$, and we want this to be the set of ``ways of composing cells in the configuration of $\alpha$''.  Note that an $m$-pasting diagram may involve composition along any dimension of cell up to $m-1$, and according to our iterative operadic theory, each of these is parametrised by a different operad.  Table~\ref{bigtable} indicates to us how to find all the composites of a given pasting diagram intuitively.  For example, for the following pasting diagram

\vspace{-6pt}\[
       \def\objectstyle{\scriptstyle}
       
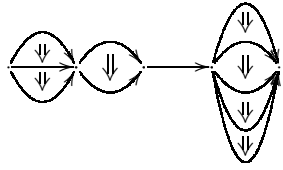
\xy
   (18,0)*{\cdot}="1";
   (27,0)*{\cdot}="2";
       {\ar@/^.8pc/ "1";"2"};
       {\ar@/_.8pc/ "1";"2"};
       {\ar@/^2pc/ "1";"2"};
       {\ar@/_2pc/ "1";"2"};
       {\ar@/_3pc/ "1";"2"};
       {\ar@{=>} (22.5,1.5)*{};(22.5,-1.5)*{}} ;
       {\ar@{=>} (22.5,7.25)*{};(22.5,4.75)*{}} ;
       {\ar@{=>} (22.5,-4.75)*{};(22.5,-7.25)*{}} ;
       {\ar@{=>} (22.5,-9.25)*{};(22.5,-11.75)*{}} ;
       (-9,0)*{\cdot}="1";
       (0,0)*{\cdot}="2";
           {\ar "1";"2"};
           {\ar@/^1.1pc/ "1";"2"};
           {\ar@/_1.1pc/ "1";"2"};
           {\ar@{=>} (-4.5,3)*{};(-4.5,.75)*{}} ;
           {\ar@{=>} (-4.5,-.75)*{};(-4.5,-3)*{}} ;
   (0,0)*{\cdot}="1";
   (9,0)*{\cdot}="2";
       {\ar@/^.8pc/ "1";"2"};
       {\ar@/_.8pc/ "1";"2"};
       {\ar@{=>} (4.5,1.75)*{};(4.5,-1.75)*{}} ;
   (9,0)*{\cdot}="1";
   (18,0)*{\cdot}="2";
    {\ar "1";"2"};
\endxy\]
we have four ``columns'' of 1-composites, the results of which are composed (horizontally) by a 4-ary 0-composition.  We need a parametrising element for \emph{each} part of the composition, thus
\begin{itemize}
\item for the first column, we have a 2-ary 1-composite of 2-cells, so consulting our table we find that this is parametrised by a 0-cell of $P_{n-2}(2)$,
\item for the second column, we have a 1-ary 1-composite, so this is parametrised by a 0-cell of $P_{n-2}(1)$,
\item similarly for the third column we need a 0-cell of $P_{n-2}(0)$,
\item for the last column we need a 0-cell of $P_{n-2}(4)$, and finally
\item the 0-composition is parametrised by a 1-cell of $P_{n-1}(4)$.
\end{itemize}
So, writing the $m$-cells of $P_i(k)$ as $P_i(k)_m$, we see that the set of ``ways of composing'' the above diagram should be given by
\[P_{n-1}(4)_1 \times P_{n-2}(2)_0 \times P_{n-2}(1)_0 \times P_{n-2}(0)_0 \times P_{n-2}(4)_0.\]
Of course, to prove this rigorously we use a much more abstract argument.  We use the fact that a globular operad is given by a cartesian monad $Q$, equipped with a cartesian natural transformation $Q \Rightarrow T$,
where $T$ is the free strict $\omega$-category monad on the category \cat{GSet} of globular sets, or the free strict $n$-category monad if we are dealing with $n$-dimensional globular operads.
Thus, to construct the operad for iterative operadic $n$-categories, we construct its associated monad. We follow Leinster's method for constructing the monad for \emph{strict} $n$-categories, which proceeds by induction using:
\begin{itemize}
\item the free strict $(n-1)$-category monad, $T^{(n-1)}$,
\item the free \cl{V}-category monad $\cat{fc}_{\cl{V}}$ on \cat{\cl{V}-Gph}, with $\cl{V} = \cat{$(n-1)$-GSet}$, and
\item a distributive law governing their interaction.
\end{itemize}
So we can construct the free strict $n$-category on an $n$-globular set in the following steps:
\begin{enumerate}
\item construct $i$-composites for $i \geq 1$, using $T^{(n-1)}$, and then
\item construct 0-composites freely, using the monad $\cat{fc}_{\cl{V}}$ as above.
\end{enumerate}
The distributive law comes from the interchange laws of 0-composition and $i$-composition, for all $i \geq 1$; the categories \cat{$n$-GSet} of $n$-dimensional globular sets appear as a result of iterating the \cl{V}-graph construction, starting with $\cl{V} =\cat{Set}$ (Lemma~\ref{lemmathreepointtwo}).

We can copy this construction for our weak enrichment, using:
\begin{itemize}
\item the free \emph{weak} $(n-1)$-category monad, $Q^{(n-1)}$ (by induction),
\item the free $(\cl{V},P)$-category monad $\cat{fc}_{(\cl{V},P)}$, with $\cl{V} = \cat{$(n-1)$-GSet}$ and $P$ the underlying globular set operad of $P_{n-1}$, and
\item a distributive law governing their interaction.
\end{itemize}
In this case the distributive law comes from a sort of ``parametrised interchange law'' which we will explain at the end of  Section~\ref{fourpointtwo}.  Note that the presence of this law means that not \emph{all} Batanin $n$-categories can be achieved in this way; we will discuss this issue more at the end of the Introduction.

To complete our main result we need a cartesian natural transformation
\[Q^{(n)} \Rightarrow T^{(n)}\]
for each $n \geq 0$; this is induced by the canonical operad morphisms from each $P_i$ to the terminal operad.  So we have half of our main result: we have for each $n \geq 0$ a globular operad whose algebras are the $n$-categories we started with.  In Proposition~\ref{propositionfourpointfive} we show that the formula obtained by this abstract argument is indeed the one we first thought of by considering the entries in Table~\ref{bigtable}.  This is useful for the purposes of satisfying our intuition, but is also useful to prove the rest of the result: that the contractibility of $Q^{(n)}$ corresponds to the contractibility of the $P_i$.

In Section~\ref{fourpointtwo} we briefly discuss the unravelling of the above inductive argument.  We find that we have for each $0 \leq i < n$ a monad $Q^{(n)}_i$ for ``composition along bounding $i$-cells''. That is, we can isolate each sort of composition and build it freely, parametrised by the relevant operad $P_{n-i-1}$. We then find that we have for each $n \geq 3$ a ``distributive series of monads'' as in \cite{che17}
\[Q^{(n)}_0, \ldots, Q^{(n)}_{n-1}\]
which is exactly analogous to the distributive series of monads giving \emph{strict} $n$-categories.

In Section~\ref{five} we apply the results of the rest of the work to Trimble's original definition (that is, involving his operad $E$), with the aim of relating it to the operad Batanin uses to take fundamental groupoids of a space.  The idea is that any topological space $X$ has a natural underlying globular set $GX$ whose
\begin{itemize}
\item 0-cells are the points of $X$
\item 1-cells are the paths of $X$
\item 2-cells are the homotopies between paths of $X$
\item 3-cells are the homotopies between homotopies between paths of $X$\\[4pt]
\hspace*{2em}$\vdots$
\end{itemize}
and that these should be the cells of the fundamental $\omega$-groupoid of $X$; for the fundamental $n$-groupoid we need to take homotopy classes at dimension $n$.  In order to exhibit this as an $n$-groupoid we first need to equip it with the structure of an $n$-category.  In Batanin's theory, this means we must find a contractible globular operad for which this globular set is an algebra.  (We will not be concerned with showing it is an $n$-groupoid here.)

Batanin constructs a contractible globular operad $K$ with a canonical action on the underlying globular set of any space.  Given an $m$-pasting diagram $\alpha$ he defines the operations of $K$ of arity $\alpha$ to be the continuous, boundary-preserving maps from the topological $m$-disc to the geometric realisation of $\alpha$.  These can be thought of as higher-dimensional reparametrising maps -- exactly the maps we would need to turn a pasting diagram of cells in $GX$ back into a single cell of $X$.

So in Section~\ref{five} we consider the following process.
\begin{enumerate}
\item Start with Trimble's topological operad $E$.
\item Use $E$ to make Trimble's original iterative operadic theory of $n$-categories.
\item Apply the constructions of Section~\ref{four} to produce an associated globular operad for this theory.
\item Embed this operad in Trimble's operad $K$ for fundamental $n$-groupoids.
\end{enumerate}
The aim of Section~\ref{five}, then, is to construct this embedding.  This works, essentially, by taking Trimble's linear reparametrising maps
\[[1] \lra [k]\]
and letting them act naturally on $m$-cubes; we then quotient the cubes to form $m$-discs.  Constructing an operad morphism to $K$ is then straightforward; we will not prove here that it is an embedding, but study its properties in a future work.

Note that throughout this work we will take the natural numbers \bb{N} to include 0.

\subsection*{Remarks on weak vs strict interchange}

In the definition of a bicategory as a category enriched in categories, the interchange law corresponds to the functoriality of the composition functor.  The same is true in general in an $n$-category defined by enrichment, and thus the strictness of interchange corresponds to the strictness of the functors we are using in our enrichment.

Trimble's definition only defines \emph{strict} functors between $n$-categories, and thus the definition may be thought of as having strict interchange laws, although they are parametrised in a slightly subtle way (see end of Section~\ref{fourpointtwo} ).  This might be considered to be ``too strict'' and indeed until recently attention was focused on ``fully weak'' definitions.  Trimble's intention was explicitly \emph{not} to define the most weak possible notion of $n$-category; he called his version ``flabby $n$-categories" rather than weak ones.  Rather, his stated aim was a definition that would \emph{naturally} yield fundamental $n$-groupoids of spaces, and his definition certainly achieves that -- the fundamental $n$-groupoid functor is an inherent part of the definition.

With regard to interchange, the point is that fundamental $n$-groupoids \emph{do} have strict interchange (see for example \cite{lei7}).  They do not, however, have strict units.  While strict 2-groupoids do model homotopy 2-types, Simpson proved in \cite{sim3} that strict 3-groupoids are too strict to model homotopy 3-types; in the same work he conjectures that having weak units is weak enough to model $n$-types for all $n$.  This conjecture has been proved at dimension 3 by Joyal and Kock in \cite{jk1}.  All this indicates that the combination of ``strict interchange and weak units'' is worth studying.

Furthermore, this is related to the question of coherence.  While not every weak 3-category is equivalent to a completely strict one, the coherence theorem of \cite{gps1} tells us that every weak 3-category is equivalent to a Gray-category, which can be thought of as a semi-strict 3-category in which everything is strict \emph{except} interchange.  It is generally believed that an analogous result should be true for higher dimensions.  However, there is a different sort of semi-strict $n$-category whose candidacy for a coherence theorem is effectively highlighted and supported by the work of Joyal and Kock: that is, a semi-strict $n$-category in which everything is strict \emph{except units}.  They have already proved in \cite{jk1} that for $n = 3$ this is enough to produce braided monoidal categories in the doubly degenerate case, which is effectively the content of the coherence theorem for 3-categories.

Iterative operadic $n$-categories, then, may be thought of as this latter form of semi-strict $n$-category.  As such they should prove useful for the study of both homotopy types and coherence.

\subsection*{Acknowledgements}

Much of this work was completed with the help of the very conducive research environments of the Laboratoire J. A. Dieudonn\'e at the Universit\'e de Nice Sophia-Antipolis, and the Fields Institute, Toronto, for which I am very grateful. I would also like to thank Andr\'e Joyal and Tom Leinster for useful discussions.

\section{Trimble-like theories of $n$-category}\label{one}

We begin this section with a concise account of Trimble's original definition of $n$-category as presented in \cite{lei7}.  Then, we propose a generalisation of this theory which allows for the use of operads other than Trimble's original operad $E$.  This more general framework will also serve as further explanation of Trimble's original definition.  It is this generalisation that we will refer to as ``Trimble-like'', or iterative operadic.

\subsection{Trimble's original definition}\label{trimbleorig}\label{onepointone}

Trimble's definition of $n$-category proceeds by iterated enrichment.  It is well-known that strict $n$-categories can be defined by repeated enrichment, that is, for $n \geq 1$, a strict $n$-category is precisely a category enriched in strict $(n-1)$-categories.  However, for \emph{weak} $n$-categories a notion of ``weak enrichment'' is required.

Trimble weakens the notion of enrichment using an operad action.  The operad he uses is a specific one, which we now define.

\begin{mydefinition}\label{operade}

We define the (classical, non-symmetric) operad $E$ in \cat{Top} by setting $E(k)$ to be the space of continuous endpoint-preserving maps
    \[ [0,1] \lra [0,k] .\]
for each $k \geq 0$.  We will write the closed interval $[0,k]$ as $[k]$, and we will often write $[0,1]$ as $I$. The composition maps
\[E(m) \times E(k_1) \times \cdots \times E(k_m) \lra E(k_1 + \cdots + k_m) \]
are given by reparametrisation and the unit is given by the identity map $[1] \lra [1]$ in $E(1)$.
\end{mydefinition}

\begin{myremarks}  The operad $E$ defined above has the following two crucial properties.

\begin{enumerate}

\item $E$ has a natural action on path spaces.  That is, given a space $X$ and points $x_0, \ldots, x_k \in X$ we have a canonical map
    \[E(k) \times X(x_{k-1}, x_k) \times \cdots \times X(x_0, x_1) \lra X(x_0, x_k)\]
    compatible with the operad composition.  Note that this action will be crucial for making the induction step in the definition.

\item For each $k \geq 0$, the space $E(k)$ is contractible.  This is what will give \emph{coherence} for the $n$-categories we define; however from a technical point of view the induction will not depend on this property of $E$.

\end{enumerate}
\end{myremarks}

We are now ready to make the definition.  We will simultaneously define, for each $n \geq 0$
\begin{itemize}
\item a finite product category \cat{$n$-Cat} of weak $n$-categories, and
\item a product-preserving functor $\Pi_n: \cat{Top} \lra \cat{$n$-Cat}$, which is intended to be the fundamental $n$-groupoid functor.
\end{itemize}
The idea is to use the fundamental $n$-groupoid of each $E(k)$ to parametrise $k$-ary composition in an $(n+1)$-category.

\begin{mydefinition}
First set $\cat{$0$-Cat}=\cat{Set}$ and define $\Pi_0$ to be the functor
\[\cat{Top} \lra \cat{Set}\]
which sends a space $X$ to its set of connected components.  Observe that $\Pi_0$ preserves products.

For $n \geq 1$ a \demph{weak $n$-category} $A$ consists of
\begin{itemize}
   \item objects: a set $A_0$,
   \item hom-$(n-1)$-categories: for all $a,a'\in A_0$ an $(n-1)$-category $A(a,a')$, and
   \item $k$-ary composition: for all $k\geq 0$ and $a_0,\dots ,a_k\in
   A_0$, an $(n-1)$-functor
\[
   \gamma :\Pi_{n-1}\big(E(k)\big)\times A(a_{k-1},a_k) \times \cdots \times A(a_0,a_1)
   \lra A(a_0,a_k)
\]
\end{itemize}
compatible with the operadic composition of $E$.

Given such $n$-categories $A$ and $B$, an \demph{$n$-functor} (or just \demph{functor}) $A \map{F} B$ consists of
\begin{itemize}
   \item  on objects: a function $F : A_0 \lra B_0$, and
   \item  on morphisms: for all $a,a'\in A_0$ an
   $(n-1)$-functor
   \[
      F: A(a,a') \lra B(Fa,Fa')
   \]
\end{itemize}
satisfying ``functoriality" --- for all $k\geq 0$ and $a_0,\ldots
,a_k \in A_0$ the following diagram commutes:
\[\xy
(-20,10)*+{\Pi_{n-1}\big(E(k)\big)\times A(a_{k-1},a_k)\times
\cdots \times A(a_{0},a_1)}="tl"; (40,10)*+{A(a_0,a_k)}="tr";
(-20,-10)*+{\Pi_{n-1}\big(E(k)\big)\times B(Fa_{k-1},Fa_k)\times
\cdots \times B(Fa_{0},Fa_1)}="bl";
(40,-10)*+{B(Fa_0,Fa_k).}="br";
{\ar^>>>>>>>>{\gamma} "tl";"tr"};
{\ar_{1 \times F \times \cdots \times F} "tl";"bl"};
{\ar^{F} "tr";"br"};
{\ar^>>>>>{\gamma} "bl";"br"};
\endxy
\]

\noindent We write \cat{$n$-Cat} for the category of weak $n$-categories and their functors and observe that it has finite products.

We also define a functor
   \[\Pi_n: \cat{Top} \lra \cat{$n$-Cat}\]
as follows.  Given a space $X$ we define an $n$-category $\Pi_n X$ by
\begin{itemize}
   \item objects: $(\Pi_n X)_0$ is the underlying set of $X$,
   \item hom-$(n-1)$-categories: $(\Pi_n X)(x,x')=\Pi_{n-1}(X(x,x'))$, and
   \item composition: we use the action of $E$ on path spaces and
   the fact that $\Pi_{n-1}$ preserves products to make the following composition functor
   \[\xy
(0,15)*+{\Pi_{n-1}\big(E(k)\big)\times\Pi_{n-1}\big(X(x_{k-1},x_k)\big)\times
\cdots \times \Pi_{n-1}\big(X(x_{0},x_1)\big)}="tl";
(0,0)*+{\Pi_{n-1}\big(E(k) \times X(x_{k-1},x_k)\times \cdots
\times X(x_{0},x_1)\big)}="bl";
(0,-20)*+{\Pi_{n-1}\big(X(x_0,x_k)\big)}="br";
{\ar^<<<<<{\iso \mbox{\upshape \textsf{\ \ \ $\Pi_{n-1}$ preserves products}}}"tl";"bl"};
{\ar^{\mbox{\upshape \ \ \textsf{$\Pi_{n-1}$ of the action of $E$
on path spaces}}} "bl";"br"};
\endxy.
\]

\end{itemize}
The action of $\Pi_n$ on morphisms is defined in the obvious way.  Finally observe that $\Pi_n$ preserves products, so the induction goes through.
\end{mydefinition}

\begin{myremarks} \hspace*{1em}
\begin{enumerate}
\item Note that the functors defined here are ``strict functors'', so the enrichment gives strict interchange even though everything else about the definition is weak.
\item Some further explanation about compatibility with operad operations will be given in the next section.
\end{enumerate}
\end{myremarks}

\subsection{A more general version of Trimble's definition}\label{onepointtwo}

Trimble's definition relies on use of the topological operad $E$, but in fact the operads used to parametrise composition are the operads $\Pi_n(E)$ given by
\[\Pi_n(E)(k) = \Pi_n(E(k)) \in \cat{$n$-Cat}.\]
The fact that the functor $\Pi_n : \cat{Top} \lra \cat{$n$-Cat}$ preserves products ensures that $\Pi_n$ takes operads to operads (since this is true of any lax monoidal functor), so $\Pi_n(E)$ defined in this way is indeed an operad.

In this section we define an ``iterative operadic theory of $n$-categories'' to be given by, for all $n$, a category \cat{$n$-Cat} of $n$-categories and an operad $P_n$ in \cat{$n$-Cat}, such that \cat{$(n+1)$-Cat} is the category of ``categories enriched in \cat{$n$-Cat} weakened by $P_n$''.   In fact we will make a general definition of ``$(\cl{V},P)$-category'' where \cl{V} is the monoidal category in which we are enriching, and $P$ is an operad in \cl{V} which we are using to parametrise composition.  This is also called a ``categorical $P$-algebra'' or ``$P$-category'' \cite[Section 10.1]{lei8}, and can be thought of as a cross between a \cl{V}-category and a $P$-algebra; we will see that it generalises both of these notions.

\begin{mydefinition}

Given a category \cl{V}, a \demph{\cl{V}-graph} $A$ is given by

\begin{itemize}
\item a set $A_0$ of objects, and
\item for every pair of objects $a,a'$, a hom-object $A(a,a') \in \cl{V}$.
\end{itemize}
A morphism $F: A \lra B$ of \cl{V}-graphs is given by

\begin{itemize}
\item a function $F:A_0 \lra B_0$, and
\item for every pair of objects $a,a'$, a morphism $A(a,a') \lra B(Fa, Fa') \in \cl{V}$.
\end{itemize}
\cl{V}-graphs and their morphisms form a category \cat{\cl{V}-Gph}.
\end{mydefinition}

\begin{myremark}
Note that \cat{\cl{V}-Gph} inherits many of the properties of \cl{V}.  We will use the fact that it is cartesian if \cl{V} is, with pullbacks given componentwise.
\end{myremark}

\begin{mydefinition}
Let \cl{V} be a symmetric monoidal category and $P$ an operad in \cl{V}.  A \demph{$(\cl{V},P)$-category} $A$ is given by
\begin{itemize}
\item a \cl{V}-graph $A$, equipped with
\item for all $k \geq 0$ and $a_0, \ldots, a_k \in A_0$ a composition morphism
\[\gamma : P(k) \otimes A(a_{k-1}, a_k) \otimes \cdots \otimes A(a_0, a_1) \lra A(a_0, a_k)\]
\end{itemize}
in \cl{V}, compatible with the composition of the operad.  Note that composition for the case $k=0$ is to be interpreted as, for all $a \in A_0$ a morphism
\[P(0) \lra A(a,a).\]

\noindent A \demph{morphism $F:A \lra B$ of $(\cl{V},P)$-categories} is a morphism of the underlying \cl{V}-graphs such that the following diagram commutes
%
\[\xy
(0,25)*+{P(k) \otimes A(a_{k-1}, a_k) \otimes \cdots \otimes A(a_0,a_1)}="1";
(60,25)*+{A(a_0, a_k)}="2";
(0,0)*+{P(k) \otimes B(Fa_{k-1}, Fa_k) \otimes \cdots \otimes B(Fa_0,Fa_1)}="3";
(60,0)*+{B(Fa_0, Fa_k)}="4";
{\ar^>>>>>>>>>>>>{\gamma} "1";"2"};
{\ar^{F} "2";"4"};
{\ar_{1 \otimes F \otimes \cdots \otimes F} "1";"3"};
{\ar^>>>>>>>>{\gamma} "3";"4"};
\endxy\]
Then $(\cl{V},P)$-categories and their morphisms form a category \cat{$(\cl{V},P)$-Cat}.
\end{mydefinition}

\noindent Note that the compatibility condition in the definition of $(\cl{V},P)$-category can be sketched in pictures as follows.  We represent elements of the operad as operations with multiple inputs and one output; for example an element of $P(3)$ is represented as:
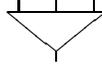
\[\xy
(13.5,40)*{}="11";
(26.5,40)*{}="12";
{\ar@{-}^{} "11";"12"};
(20,34.8)*{}="16";
(32.5,35)*{}="17";
(20,33.5)*{}="16b";
{\ar@{-}^{} "11";"16"};
{\ar@{-}^{} "12";"16"};
(15,41.5)*{}="22";
(15,40)*{}="23";
(20,41.5)*{}="24";
(20,40)*{}="25";
(25,41.5)*{}="26";
(25,40)*{}="27";
{\ar@{-}^{} "22";"23"};
{\ar@{-}^{} "24";"25"};
{\ar@{-}^{} "26";"27"};
{\ar@{-}^{} "16";"16b"};
\endxy\]
Operadic composition is then represented by:
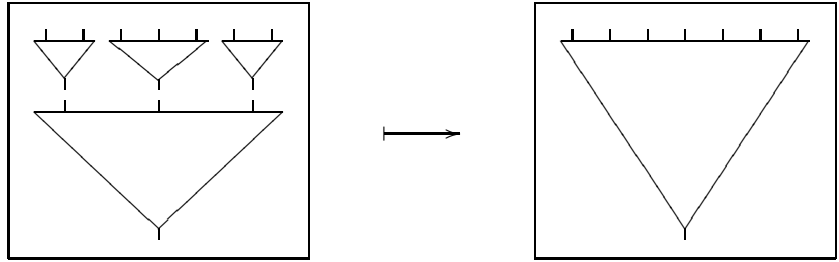
\[\xy
{\ar@{|->}^{} (30,65);(40,65)};
(0,60)*
{\xy   
(3.5,40)*{}="9";
(11.5,40)*{}="10";
(13.5,40)*{}="11";
(26.5,40)*{}="12";
(28.5,40)*{}="13";
(36.5,40)*{}="14";
{\ar@{-}^{} "9";"10"};
{\ar@{-}^{} "11";"12"};
{\ar@{-}^{} "13";"14"};
(7.5,35)*{}="15";
(20,34.8)*{}="16";
(32.5,35)*{}="17";
(7.5,33.5)*{}="15b";
(20,33.5)*{}="16b";
(32.5,33.5)*{}="17b";
(7.5,32)*{}="15c";
(20,32)*{}="16c";
(32.5,32)*{}="17c";
(7.5,30.5)*{}="15d";
(20,30.5)*{}="16d";
(32.5,30.5)*{}="17d";
{\ar@{-}^{} "9";"15"};
{\ar@{-}^{} "10";"15"};
{\ar@{-}^{} "11";"16"};
{\ar@{-}^{} "12";"16"};
{\ar@{-}^{} "13";"17"};
{\ar@{-}^{} "14";"17"};
(5,41.5)*{}="18";
(5,40)*{}="19";
(10,41.5)*{}="20";
(10,40)*{}="21";
(15,41.5)*{}="22";
(15,40)*{}="23";
(20,41.5)*{}="24";
(20,40)*{}="25";
(25,41.5)*{}="26";
(25,40)*{}="27";
(30,41.5)*{}="28";
(30,40)*{}="29";
(35,41.5)*{}="30";
(35,40)*{}="31";
{\ar@{-}^{} "18";"19"};
{\ar@{-}^{} "20";"21"};
{\ar@{-}^{} "22";"23"};
{\ar@{-}^{} "24";"25"};
{\ar@{-}^{} "26";"27"};
{\ar@{-}^{} "28";"29"};
{\ar@{-}^{} "30";"31"};
{\ar@{-}^{} "15";"15b"};
{\ar@{-}^{} "16";"16b"};
{\ar@{-}^{} "17";"17b"};
{\ar@{-}^{} "15c";"15d"};
{\ar@{-}^{} "16c";"16d"};
{\ar@{-}^{} "17c";"17d"};
(3.5,30.5)*{}="32";
(36.5,30.5)*{}="33";
(20,15)*{}="34";
(20,13.5)*{}="34b";
{\ar@{-}^{} "32";"33"};
{\ar@{-}^{} "32";"34"};
{\ar@{-}^{} "33";"34"};
{\ar@{-}^{} "34";"34b"};
(0,45)*{}="35";
(40,45)*{}="36";
(40,11)*{}="37";
(0,11)*{}="38";
{\ar@{-}^{} "35";"36"};
{\ar@{-}^{} "36";"37"};
{\ar@{-}^{} "37";"38"};
{\ar@{-}^{} "38";"35"};
\endxy}; 
(70,60)*
{\xy      
(3.5,40)*{}="9";
(11.5,40)*{}="10";
(13.5,40)*{}="11";
(26.5,40)*{}="12";
(28.5,40)*{}="13";
(36.5,40)*{}="14";
(20,15)*{}="34";
(20,13.5)*{}="34b";
{\ar@{-}^{} "9";"14"};
{\ar@{-}^{} "9";"34"};
{\ar@{-}^{} "14";"34"};
(5,41.5)*{}="18";
(5,40)*{}="19";
(10,41.5)*{}="20";
(10,40)*{}="21";
(15,41.5)*{}="22";
(15,40)*{}="23";
(20,41.5)*{}="24";
(20,40)*{}="25";
(25,41.5)*{}="26";
(25,40)*{}="27";
(30,41.5)*{}="28";
(30,40)*{}="29";
(35,41.5)*{}="30";
(35,40)*{}="31";
{\ar@{-}^{} "18";"19"};
{\ar@{-}^{} "20";"21"};
{\ar@{-}^{} "22";"23"};
{\ar@{-}^{} "24";"25"};
{\ar@{-}^{} "26";"27"};
{\ar@{-}^{} "28";"29"};
{\ar@{-}^{} "30";"31"};
{\ar@{-}^{} "34";"34b"};
(0,45)*{}="35";
(40,45)*{}="36";
(40,11)*{}="37";
(0,11)*{}="38";
{\ar@{-}^{} "35";"36"};
{\ar@{-}^{} "36";"37"};
{\ar@{-}^{} "37";"38"};
{\ar@{-}^{} "38";"35"};
\endxy}; 
\endxy
\]
We represent elements of $A(a,a')$ as arrows, and elements of \[A(a_{k-1}, a_k) \otimes \cdots \otimes A(a_0, a_1)\] as strings of composable arrows
\[\xy
(5,43)*{\rightarrow};
(10,43)*{\rightarrow};
(15,43)*{\rightarrow};
(20,43)*{\rightarrow};
(25,43)*{\rightarrow};
(30,43)*{\rightarrow};
(35,43)*{\rightarrow};
\endxy\]
without specifically labelling the endpoints $a_i$.  Then the composition in our $(\cl{V}, P)$-category $A$ is represented by:
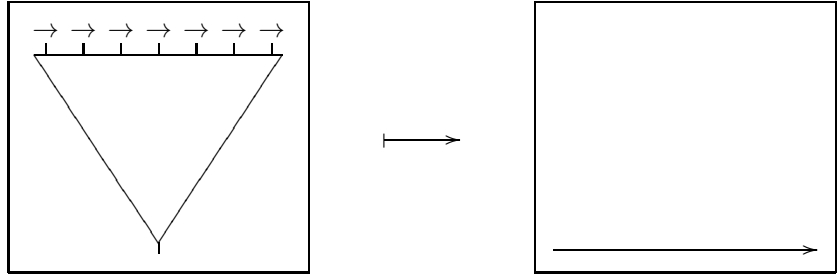
\[\xy
{\ar@{|->}^{} (30,5);(40,5)};
(0,0)*
{\xy      
(5,43)*{\rightarrow};
(10,43)*{\rightarrow};
(15,43)*{\rightarrow};
(20,43)*{\rightarrow};
(25,43)*{\rightarrow};
(30,43)*{\rightarrow};
(35,43)*{\rightarrow};
(3.5,40)*{}="9";
(11.5,40)*{}="10";
(13.5,40)*{}="11";
(26.5,40)*{}="12";
(28.5,40)*{}="13";
(36.5,40)*{}="14";
(20,15)*{}="34";
(20,13.5)*{}="34b";
{\ar@{-}^{} "9";"14"};
{\ar@{-}^{} "9";"34"};
{\ar@{-}^{} "14";"34"};
(5,41.5)*{}="18";
(5,40)*{}="19";
(10,41.5)*{}="20";
(10,40)*{}="21";
(15,41.5)*{}="22";
(15,40)*{}="23";
(20,41.5)*{}="24";
(20,40)*{}="25";
(25,41.5)*{}="26";
(25,40)*{}="27";
(30,41.5)*{}="28";
(30,40)*{}="29";
(35,41.5)*{}="30";
(35,40)*{}="31";
{\ar@{-}^{} "18";"19"};
{\ar@{-}^{} "20";"21"};
{\ar@{-}^{} "22";"23"};
{\ar@{-}^{} "24";"25"};
{\ar@{-}^{} "26";"27"};
{\ar@{-}^{} "28";"29"};
{\ar@{-}^{} "30";"31"};
{\ar@{-}^{} "34";"34b"};
(0,47)*{}="35";
(40,47)*{}="36";
(40,11)*{}="37";
(0,11)*{}="38";
{\ar@{-}^{} "35";"36"};
{\ar@{-}^{} "36";"37"};
{\ar@{-}^{} "37";"38"};
{\ar@{-}^{} "38";"35"};
\endxy}; 
(70,0)*
{\xy    
(2.5,14)*{}="43";
(37.5,14)*{}="44";
{\ar^{} "43";"44"};
(0,47)*{}="35";
(40,47)*{}="36";
(40,11)*{}="37";
(0,11)*{}="38";
{\ar@{-}^{} "35";"36"};
{\ar@{-}^{} "36";"37"};
{\ar@{-}^{} "37";"38"};
{\ar@{-}^{} "38";"35"};
\endxy}; 
\endxy
\]
showing that we take a string of $k$ composable arrows together with an element of $P(k)$ and produce a single arrow as a result.  Then the compatibility is represented by the commutativity of the following diagram:
\[\xy
{\ar@{|->}^{\gamma} (30,65);(40,65)};   
{\ar@{|->}_{\gamma} (30,5);(40,5)};    
{\ar@{|->}_{\mbox{\textsf{composition in $P$}}} (0,40);(0,30)};   
{\ar@{|->}^{\gamma} (70,40);(70,30)};  
(0,60)*
{\xy   
(5,43)*{\rightarrow};
(10,43)*{\rightarrow};
(15,43)*{\rightarrow};
(20,43)*{\rightarrow};
(25,43)*{\rightarrow};
(30,43)*{\rightarrow};
(35,43)*{\rightarrow};
(3.5,40)*{}="9";
(11.5,40)*{}="10";
(13.5,40)*{}="11";
(26.5,40)*{}="12";
(28.5,40)*{}="13";
(36.5,40)*{}="14";
{\ar@{-}^{} "9";"10"};
{\ar@{-}^{} "11";"12"};
{\ar@{-}^{} "13";"14"};
(7.5,35)*{}="15";
(20,34.8)*{}="16";
(32.5,35)*{}="17";
(7.5,33.5)*{}="15b";
(20,33.5)*{}="16b";
(32.5,33.5)*{}="17b";
(7.5,32)*{}="15c";
(20,32)*{}="16c";
(32.5,32)*{}="17c";
(7.5,30.5)*{}="15d";
(20,30.5)*{}="16d";
(32.5,30.5)*{}="17d";
{\ar@{-}^{} "9";"15"};
{\ar@{-}^{} "10";"15"};
{\ar@{-}^{} "11";"16"};
{\ar@{-}^{} "12";"16"};
{\ar@{-}^{} "13";"17"};
{\ar@{-}^{} "14";"17"};
(5,41.5)*{}="18";
(5,40)*{}="19";
(10,41.5)*{}="20";
(10,40)*{}="21";
(15,41.5)*{}="22";
(15,40)*{}="23";
(20,41.5)*{}="24";
(20,40)*{}="25";
(25,41.5)*{}="26";
(25,40)*{}="27";
(30,41.5)*{}="28";
(30,40)*{}="29";
(35,41.5)*{}="30";
(35,40)*{}="31";
{\ar@{-}^{} "18";"19"};
{\ar@{-}^{} "20";"21"};
{\ar@{-}^{} "22";"23"};
{\ar@{-}^{} "24";"25"};
{\ar@{-}^{} "26";"27"};
{\ar@{-}^{} "28";"29"};
{\ar@{-}^{} "30";"31"};
{\ar@{-}^{} "15";"15b"};
{\ar@{-}^{} "16";"16b"};
{\ar@{-}^{} "17";"17b"};
{\ar@{-}^{} "15c";"15d"};
{\ar@{-}^{} "16c";"16d"};
{\ar@{-}^{} "17c";"17d"};
(3.5,30.5)*{}="32";
(36.5,30.5)*{}="33";
(20,15)*{}="34";
(20,13.5)*{}="34b";
{\ar@{-}^{} "32";"33"};
{\ar@{-}^{} "32";"34"};
{\ar@{-}^{} "33";"34"};
{\ar@{-}^{} "34";"34b"};
(0,47)*{}="35";
(40,47)*{}="36";
(40,11)*{}="37";
(0,11)*{}="38";
{\ar@{-}^{} "35";"36"};
{\ar@{-}^{} "36";"37"};
{\ar@{-}^{} "37";"38"};
{\ar@{-}^{} "38";"35"};
\endxy}; 
(70,60)*
{\xy     
(7.5,32)*{}="15c";
(20,32)*{}="16c";
(32.5,32)*{}="17c";
(7.5,30.5)*{}="15d";
(20,30.5)*{}="16d";
(32.5,30.5)*{}="17d";
{\ar@{-}^{} "15c";"15d"};
{\ar@{-}^{} "16c";"16d"};
{\ar@{-}^{} "17c";"17d"};
(3.5,30.5)*{}="32";
(36.5,30.5)*{}="33";
(20,15)*{}="34";
(20,13.5)*{}="34b";
{\ar@{-}^{} "32";"33"};
{\ar@{-}^{} "32";"34"};
{\ar@{-}^{} "33";"34"};
{\ar@{-}^{} "34";"34b"};
(2.5,34)*{\ }="39";
(12.5,34)*{\ }="40";
(27.5,34)*{\ }="41";
(37.5,34)*{\ }="42";
{\ar^{} "39";"40"};
{\ar^{} "40";"41"};
{\ar^{} "41";"42"};
(0,47)*{}="35";
(40,47)*{}="36";
(40,11)*{}="37";
(0,11)*{}="38";
{\ar@{-}^{} "35";"36"};
{\ar@{-}^{} "36";"37"};
{\ar@{-}^{} "37";"38"};
{\ar@{-}^{} "38";"35"};
\endxy};  
(0,0)*
{\xy      
(5,43)*{\rightarrow};
(10,43)*{\rightarrow};
(15,43)*{\rightarrow};
(20,43)*{\rightarrow};
(25,43)*{\rightarrow};
(30,43)*{\rightarrow};
(35,43)*{\rightarrow};
(3.5,40)*{}="9";
(11.5,40)*{}="10";
(13.5,40)*{}="11";
(26.5,40)*{}="12";
(28.5,40)*{}="13";
(36.5,40)*{}="14";
(20,15)*{}="34";
(20,13.5)*{}="34b";
{\ar@{-}^{} "9";"14"};
{\ar@{-}^{} "9";"34"};
{\ar@{-}^{} "14";"34"};
(5,41.5)*{}="18";
(5,40)*{}="19";
(10,41.5)*{}="20";
(10,40)*{}="21";
(15,41.5)*{}="22";
(15,40)*{}="23";
(20,41.5)*{}="24";
(20,40)*{}="25";
(25,41.5)*{}="26";
(25,40)*{}="27";
(30,41.5)*{}="28";
(30,40)*{}="29";
(35,41.5)*{}="30";
(35,40)*{}="31";
{\ar@{-}^{} "18";"19"};
{\ar@{-}^{} "20";"21"};
{\ar@{-}^{} "22";"23"};
{\ar@{-}^{} "24";"25"};
{\ar@{-}^{} "26";"27"};
{\ar@{-}^{} "28";"29"};
{\ar@{-}^{} "30";"31"};
{\ar@{-}^{} "34";"34b"};
(0,47)*{}="35";
(40,47)*{}="36";
(40,11)*{}="37";
(0,11)*{}="38";
{\ar@{-}^{} "35";"36"};
{\ar@{-}^{} "36";"37"};
{\ar@{-}^{} "37";"38"};
{\ar@{-}^{} "38";"35"};
\endxy}; 
(70,0)*
{\xy    
(2.5,14)*{}="43";
(37.5,14)*{}="44";
{\ar^{} "43";"44"};
(0,47)*{}="35";
(40,47)*{}="36";
(40,11)*{}="37";
(0,11)*{}="38";
{\ar@{-}^{} "35";"36"};
{\ar@{-}^{} "36";"37"};
{\ar@{-}^{} "37";"38"};
{\ar@{-}^{} "38";"35"};
\endxy}; 
\endxy
\]

\vspace{2em}
\begin{myexamples} \mbox{\ }

\begin{enumerate}

\item Put $\cl{V} = \cat{Set}$ and $P=1$ the terminal operad i.e. each $P(k)=1$.  Then a $(\cl{V}, P)$-category is just an ordinary small category.

\item Let $A$ be a $(\cl{V}, P)$-category with only one object $*$, thus only one hom-object $A(*,*) \in \cl{V}$ which by abuse of notation we write as $A$.  Then  the composition morphism for each $k$ becomes a morphism
\[P(k) \otimes A^{\otimes k} \lra A\]
and the axioms show precisely that $A$ is an algebra for the operad $P$.

\item Let $P$ be the operad defined by putting each $P(k)=I$ with composition given by the unique coherence isomorphisms in \cl{V}.  Then \cat{$(\cl{V},P)$-Cat} is equivalent to \cat{\cl{V}-Cat}, the usual category of categories enriched in \cl{V}.

\item Put $\cl{V} = \cat{Top}$ and take $P$ to be the operad $E$ used in the previous section (Definition~\ref{operade}).  Then a $(\cl{V},P)$-category can be thought of as a ``category weakly enriched in spaces''.  Any topological space is naturally a $(\cl{V},P)$-category \cite[Example 5.1.10]{lei8}.  In fact this is the precise formulation of the fact that ``$E$ has a natural action on path spaces'', and lies at the heart of why Trimble's definition of $n$-category seems natural for modelling homotopy types.

\end{enumerate}
\end{myexamples}

\begin{myremark} Note that examples (2) and (3) above show how the notion of a $(\cl{V},P)$-category is a generalisation and conflation of the notions of \cl{V}-category and $P$-algebra.  In fact, $(\cl{V},P)$-categories are precisely the algebras for a related generalised operad $\Sigma P$, as we will show in Section~\ref{twopointtwo}.
\end{myremark}

In all our examples, the tensor product in \cl{V} will be given by a categorical product; in this case \cat{$(\cl{V},P)$-Cat} also has products, and we can then iterate the construction.  This iteration gives us a candidate notion of $n$-category; it only remains to have a way of saying that such $n$-categories are ``sensible'' or coherent.  For this we will use the notion of contractibility of an operad; since we are iterating our constructions, we also need the notion of contractibility of a $(\cl{V},P)$-category.

\begin{mydefinition} \hspace*{1em}

\begin{enumerate}

\item We say that a set is \demph{contractible} if and only if it is terminal.

\item  Suppose \cl{V} is a category with a notion of contractibility, that is, we know what it means for an object of \cl{V} to be ``contractible''.  Then we say that an operad $P \in \cl{V}$ is \demph{contractible} if each $P(k)$ is contractible.

\item Suppose \cl{V} has a notion of contractibility.  Then we say a $\cl{V}$-graph is \demph{contractible} if

\begin{itemize}
\item $A_0 \neq \emptyset$, and
\item for all $a,a' \in A_0$, the hom-object $A(a,a')$ is contractible in \cl{V}.
\end{itemize}

\item We say a $(\cl{V},P)$-category is \demph{contractible} if its underlying \cl{V}-graph is contractible.

\end{enumerate}
\end{mydefinition}

\begin{myremarks}\label{remtwo} \hspace*{1em}

\begin{enumerate}

\item Note that by starting our inductive definition with the 1-element sets, we ensure that in any contractible $n$-category every homset of $n$-cells with given source and target is a 1-element set. Since we will use contractible $n$-categories to parametrise composition in an $(n+1)$-category, this is what will ensure that composition of top-dimensional cells is always strict (see Proposition~\ref{propcoh}).

\item Note that elsewhere (for example \cite{may1}) ``contractible'' is taken to mean ``weakly equivalent to the terminal object'' in a suitable model category structure; we do not address the use of model categories here.
\end{enumerate}
\end{myremarks}

We are now ready to iterate the weak enrichment construction to make $n$-categories. We are not claiming here to have made a ``new'' definition of $n$-category, nor to have improved on Trimble's remarkably elegant and concise definition.  We state the definition in the above form merely because this is the form in which we are going to use it, and we prefer to show the greatest generality in which our comparison theorem might be applied.

\begin{mydefinition}

An \demph{iterative operadic theory of $n$-categories} is given by, for all $n \geq 0$ a category $\cl{V}_n$ and a contractible operad $P_n \in \cl{V}_n$ such that

\begin{itemize}
\item $\cl{V}_0 = \cat{Set}$, and
\item for all $n \geq 0$, $\cl{V}_{n+1} = \cat{$(\cl{V}_n, P_n)$-Cat}$.
\end{itemize}
Thus $\cl{V}_n$ is the category of $n$-categories according to the theory in question.

\end{mydefinition}

Note that putting $\cl{V}_0 = \cat{Set}$ and demanding that $P_0$ be contractible means that we must have $P_0 =1$ the terminal operad.

\numarabic

\begin{myexamples} \hspace*{1em}

\begin{enumerate}

\item Trimble's original definition is an example of such a theory of $n$-categories, with $P_n = \Pi_n (E)$ for each $n \geq 0$.  The contractibility of each $E(k)$ in the topological sense ensures the contractibility of each $P_n$ in our sense.

\item Another example of such a theory is given in \cite{cg1} in which the authors propose a version of Trimble's original definition beginning with an operad in \cat{GSet} instead of in \cat{Top}.  The authors use Trimble's inductive method to produce a series of operads $P_n$ in \cat{$n$-Cat}, and they present sufficient conditions on an operad in \cat{GSet} to make the induction step work.

\item Batanin's definition of $n$-category (and variants) is a non-example.  We will see in Section~\ref{fourpointtwo} that an iterative operadic theory of $n$-categories necessarily has strict (albeit parametrised) interchange at all levels, whereas Batanin's $n$-categories allow for the possibility of weak interchange.  Thus although our main theorem will show that an iterative operadic theory can be expressed as a Batanin-type theory of $n$-categories, the converse is not true.

\item We could alternatively start with $\cl{V}_0 = \cat{Top}$ or \cat{sSet} which would give a candidate for a notion of $\omega$-category in which all cells are invertible above dimension $n$; such a structure is sometimes referred to as an $(\infty,n)$-category.   We could also start with other suitable monoidal categories with a notion of contractibility.  Note that in these cases it is not immediate that the rest of the constructions in this work will follow; we will study this in a future work.

\end{enumerate}
\end{myexamples}

%

\section{Definition via free \cl{V}-categories}\label{two}

In the previous section we demonstrated informally that a $(\cl{V},P)$-category is a ``cross'' between a $P$-algebra and a \cl{V}-category.  We now show how to derive from an operad $P$ a \emph{generalised operad} $\Sigma P$, whose algebras are precisely the $(\cl{V},P)$-categories.  This will enable the constructions of Section~\ref{four}.


\subsection{Generalised operads}\label{twopointone}

We first rapidly recall the definition of generalised operad.  The idea is to generalise the sorts of arities that the operations of an operad can have.  For a classical operad, the arities are just the natural numbers. Observing that \bb{N} is the free monoid on the terminal set, we may try applying other monads $T$ to terminal objects in other categories \cl{E}, and this works provided \cl{E} and $T$ are cartesian.  This generalisation opens up a wealth of possibilities for encoding operations whose inputs have some structure on them.  One example of such an operation is composition in an $n$-category; we will see that the ``globular operads'' in Batanin's definition are also a particular kind of generalised operad.

Generalised operads were introduced by Burroni in \cite{bur1} under the name of ``$T$-category''.  The idea was later rediscovered independently by Hermida \cite{her1} and Leinster \cite{lei11}.


\begin{mydefinition}

Let \cl{E} be a category with pullbacks (i.e. it is cartesian) and a terminal object. Let $T$ be a cartesian monad on \cl{E}, that is, the functor part preserves pullbacks, and all the naturality squares for $\eta$ and $\mu$ are pullbacks.

The category \cat{$T$-Coll} of \demph{$T$-collections} is defined to be the slice category $\cl{E}/T1$.  This is a monoidal category where the tensor product of two objects
\[\left(
\def\objectstyle{\scriptstyle}
\def\labelstyle{\scriptstyle}
\vcenter{\xymatrix @-1.0pc  { A \ar[d]^{\alpha} \\ T1  }} \right) \otimes \left(
\def\objectstyle{\scriptstyle}
\def\labelstyle{\scriptstyle}
\vcenter{\xymatrix @-1.0pc  { B \ar[d]^{\beta} \\ T1  }} \right)\]
is the left-hand edge of the following diagram
%
 \[
\xy
(0,0)*+{T1}="1";
(12,12)*+{T^2 1}="2";
(24,24)*+{TA}="3";
(36,12)*+{T1}="4";
(48,24)*+{B}="5";
(36,36)*+{\cdot}="6";
(36,31)*+{\dcorner}="";
{\ar_{\mu_1} "2";"1"};
{\ar_{T\alpha} "3";"2"};
{\ar^{T!} "3";"4"};
{\ar_{\beta} "5";"4"};
{\ar^{} "6";"3"};
{\ar^{} "6";"5"};
\endxy
\]
and the unit is the collection
\[\xymatrix  { 1 \ar[d]^{\eta_1} \\ T1 . }\]
The unit and associativity axioms follow from the fact that $T$ is a cartesian monad. \end{mydefinition}

\begin{mydefinition}  An \demph{$(\cl{E},T)$-operad} is a monoid in the monoidal category \cat{$T$-coll}.  We refer to such operads generally as \demph{generalised operads}.
\end{mydefinition}

\noindent So an $(\cl{E},T)$-operad is given by an object $P \in \cl{E}$ and a morphism
\[\xymatrix { P \ar[d]^{d} \\ T1  }\]
equipped with maps for unit
\[\left(
\def\objectstyle{\scriptstyle}
\def\labelstyle{\scriptstyle}
\vcenter{\xymatrix @-1.0pc  { 1 \ar[d]^{\eta} \\ T1  }} \right) \lra \left(
\def\objectstyle{\scriptstyle}
\def\labelstyle{\scriptstyle}
\vcenter{\xymatrix @-1.0pc  { P \ar[d]^{d} \\ T1  }} \right)\]
and multiplication
\[\left(
\def\objectstyle{\scriptstyle}
\def\labelstyle{\scriptstyle}
\vcenter{\xymatrix @-1.0pc  { P \ar[d]^{d} \\ T1  }} \right)
\otimes
\left(
\def\objectstyle{\scriptstyle}
\def\labelstyle{\scriptstyle}
\vcenter{\xymatrix @-1.0pc  { P \ar[d]^{d} \\ T1  }} \right)
\lra
\left(
\def\objectstyle{\scriptstyle}
\def\labelstyle{\scriptstyle}
\vcenter{\xymatrix @-1.0pc  { P \ar[d]^{d} \\ T1  }} \right).\]
$P$ can be thought of as the operations and $T1$ their arities; the map $d$ gives the arities of the operations.  By abuse of notation we often refer to an operad as above simply as $P$.

\vspace{2em}
\begin{myexamples} \mbox{\ }

\begin{enumerate}

\item Let $(\cl{E},T) = (\cat{Set}, \cat{id})$.  Then an $(\cl{E},T)$-operad is precisely a monoid.

\item Let $(\cl{E},T) = (\cat{Set},\mbox{``free monoid monad''})$.  Then an $(\cl{E},T)$-operad is an ordinary non-symmetric operad in \cat{Set}.

\item Let $(\cl{E},T) = (\cat{Gph}, \mbox{``free category monad''})$.  Then an $(\cl{E},T)$-operad is what Leinster calls a \cat{fc}-operad, \cat{fc} being notation for the free category monad on the category \cat{Gph} of graphs \cite[Chapter 5]{lei8}.  A key example for us will be the enriched version of this example, which we will introduce in Theorem~\ref{fcv}.

\item We will later see that a globular operad is defined to be an $(\cl{E},T)$-operad where $\cl{E}=\cat{GSet}$ and $T$ is the free strict $\omega$-category monad on \cat{GSet}; there is also an $n$-dimensional version.

\end{enumerate}
\end{myexamples}

We now turn our attention to algebras. Recall that a classical operad has an associated monad, and algebras for the operad are precisely algebras for its associated monad.  A similar result holds for $(\cl{E},T)$-operads as follows.

\begin{mydefinition}   Let $P$ be an $(\cl{E},T)$-operad.  Then the \demph{associated monad} $T_P$ on \cl{E} is defined as follows.  Given an object $A \in \cl{E}$, the object $T_PA$ is given by the following pullback
\[\xy
(0,15)*+{T_PA}="t";
   (-15,0)*+{P}="l";
   (15,0)*+{TA}="r";
   (0,-15)*+{T1}="b";
       {\ar "t";"l"};
       {\ar "t";"r"};
       {\ar "r";"b"};
       {\ar "l";"b"};
(0,10)*{\dcorner};
\endxy \]
The unit and multiplication come from the unit and multiplication of $P$.  An \demph{algebra} for the operad $P$ is then defined to be an algebra for this monad.
\end{mydefinition}


The above association of a monad to any operad extends to the following comparison result which characterises generalised operads in the form that we will construct them in Section~\ref{four}; it appears as Corollary 6.2.4 of \cite{lei8}.

Recall that a cartesian natural transformation is one whose naturality squares are all pullbacks; a cartesian monad is one whose functor part preserves pullbacks, and whose unit and multiplication are cartesian natural transformations.

%

\begin{theorem}\label{operadmonad} Let $T$ be a cartesian monad on a cartesian category \cl{E}.  Then the category of $(\cl{E},T)$-operads is equivalent to the category in which
\begin{itemize}
\item an object is a cartesian monad $P$ on \cl{E} together with a cartesian natural transformation $\alpha: P \Rightarrow T$ commuting with the monad structures, and
\item a morphism $(P,\alpha) \lra (P', \alpha')$ is a cartesian natural transformation $\theta: P \Rightarrow P'$ commuting with the monad structures and satisfying $\alpha' \circ \theta = \alpha$.
\end{itemize}
\end{theorem}

%
%


We will not give the proof of this equivalence result here, but it is worth noting that the underlying collection of an operad expressed as above is
\[\xymatrix { P1 \ar[d]^{\alpha_1} \\ T1 . } \]
By abuse of notation, we will often write $P$ for both the operad and its associated monad.

\subsection{Operad suspension}
\label{twopointtwo}


In this section we show how to take a classical operad $P$ in a suitably well-behaved category \cl{V} and construct from it an $(\cl{E},T)$-operad $\Sigma P$, for some suitable \cl{E} and $T$, such that the algebras for $\Sigma P$ are precisely the $(\cl{V},P)$-categories we defined in the previous section.  We use the notation $\Sigma$ as the construction can be viewed as a sort of ``suspension''.  This is a generalisation of the $\Sigma$ construction given by Leinster for the case $\cl{V} = \cat{Set}$ in \cite[Example 5.1.6]{lei8}.  Thus where Leinster uses $T=\cat{fc}$, the free category monad, we will use the ``free \cl{V}-category monad'' $\cat{fc}_{\cl{V}}$ which we now define.  This construction is analogous to the construction of \cat{fc}, which is itself analogous to the construction of a free monoid on a set $A$ as
\[\coprod_{k \geq 0} A^k.\]
Leinster makes this construction in \cite[Appendix F]{lei8} for any presheaf category \cl{V}.  In fact (as Leinster points out) the construction works in a much more general context; here we give it in a slightly more general context in order to gain some more examples, although for our main theorem in Section~\ref{four} we too will only need the result for presheaf categories.  First recall that the underlying data for a \cl{V}-category is a \cl{V}-graph.

%
%
%

\begin{thm}\label{fcv}
Let \cl{V} be a category with finite products and small coproducts that commute.  Then the forgetful functor
\[\cat{\cl{V}-Cat} \lra \cat{\cl{V}-Gph}\]
is monadic.  The induced  ``free \cl{V}-category monad'' $\cat{fc}_{\cl{V}}$ is cartesian and coproduct-preserving.  Its action is given as the identity on objects, and on hom-objects:
\[(\cat{fc}_{\cl{V}}A)(a,a') = \coprod_{k \geq 0,\  a=a_0, a_1, \ldots, a_{k-1}, a_k=a'} A(a_{k-1},a_k) \times \cdots \times A(a_0,a_1).\]

\end{thm}

\begin{proof}
Routine calculations.
\end{proof}

\begin{myexample} In the case $\cl{V} = \cat{Set}$ we have $\cat{fc}_{\cl{V}} = \cat{fc}$.
\end{myexample}

\begin{myexample}\label{keyexample}

The following example will be a key example for our constructions. Suppose \cl{V} satisfies the hypotheses of the above theorem. Then since \cl{V} has a terminal object 1, \cat{\cl{V}-Gph} has a terminal object which we will also write 1; it has a single object, and its single hom-object is $1 \in \cl{V}$.  Then $\cat{fc}_{\cl{V}}(1)$ has a single object and ``\bb{N}'s worth'' of morphisms.  That is, its single hom-object is
\[\coprod_{k \geq 0} 1.\]

\end{myexample}

%


\begin{mydefinition}
Let $P$ be a (classical) operad in \cl{V}.  We define the \demph{suspension}  $\Sigma P$ of $P$ to be a  $(\cat{\cl{V}-Gph}, \cat{fc}_{\cl{V}})$-operad with underlying collection
\[\xymatrix { \Sigma P \ar[d]^{} \\ \cat{fc}_{\cl{v}}1  } \]
where the \cl{V}-graph $\Sigma P$ is given by:

\begin{itemize}

\item $(\Sigma P)_0$ is a terminal set, $\{*\}$, say, and
\item $\Sigma P (*,*) = \displaystyle\coprod_{k \geq 0} P(k)$;
\end{itemize}
the morphism $\Sigma P \lra \cat{fc}_{\cl{V}}1$ is defined on hom-objects in the obvious way, by degree, so we have
\[\coprod_{k \geq 0} P(k) \lra \coprod_{k \geq 0} 1.\]
Furthermore, the unit and multiplication are constructed from the unit and multiplication for $P$.
\end{mydefinition}

\begin{myexample}
Let $\cl{V} = \cat{Set}$ and $P$ an operad in \cat{Set}.  Then $\cat{fc}_{\cl{V}} = \cat{fc}$ and $\Sigma P$ is exactly the suspension \cat{fc}-operad described in \cite[Examples 5.1.6 and 5.1.7]{lei8}.
\end{myexample}

\begin{prop}\label{sigmap}
Let \cl{V} be a cartesian category with tensor product given by product, and small coproducts commuting with pullbacks.  Then a $(\cl{V},P)$-category is an algebra for $\Sigma P$, and this extends to an equivalence of categories
\[\cat{$(\cl{V},P)$-Cat} \catequiv \mbox{$\Sigma P$-Alg}.\]
\end{prop}

\begin{prf}
An algebra for $\Sigma P$ consists of

\begin{itemize}
\item a \cl{V}-graph $A$, and
\item an algebra action $(\Sigma P)(A) \lra A$.
\end{itemize}
Here $(\Sigma P)(A)$ is given by the pullback
\[\xy
(0,15)*+{\cdot}="t";
   (-15,0)*+{\Sigma P}="l";
   (15,0)*+{\fc_{\cl{V}}(A)}="r";
   (0,-15)*+{\fc_{\cl{V}}(1)}="b";
       {\ar "t";"l"};
       {\ar "t";"r"};
       {\ar "r";"b"};
       {\ar "l";"b"};
(0,10)*{\dcorner};
\endxy \]
This pullback is given componentwise.  On object sets, the lower-lefthand map is the identity, so the the pullback has the same set of objects as $\cat{fc}_{\cl{V}}(A)$, i.e. just the objects of $A$.  For hom-objects we have for each pair of objects $a,a' \in A$ an object $(\Sigma P)(A)(a,a')$ given by a pullback of coproducts as below
\[\xy
(0,25)*+{\cdot}="t";
   (-25,0)*+{\displaystyle\coprod_{k \geq 0} P(k)}="l";
   (25,0)*+++{\makebox[13em][l]{$\displaystyle\coprod_{\begin{array}{c} {\scriptstyle k\geq 0} \\ {\scriptstyle a=a_0, \ldots, a_k=a' \in A(0)} \end{array}} \hspace{-3em}A(a_{k-1}, a_k) \times \cdots \times A(a_0, a_1)$}}="r";
   (0,-25)*+{\displaystyle\coprod_{k \geq 0} 1}="b";
       {\ar "t";"l"};
       {\ar "t";"r"};
       {\ar "r";"b"};
       {\ar "l";"b"};
(0,20)*{\dcorner};
\endxy \]
hence, since these coproducts commute with pullbacks, it is the coproduct over $k \geq 0$ of pullbacks
\[\xy
(0,20)*+{\cdot}="t";
   (-20,0)*+{P(k)}="l";
   (20,0)*++{\makebox[13em][l]{$\displaystyle\coprod_{\begin{array}{c} {\scriptstyle a=a_0, \ldots, a_k=a' \in A(0)} \end{array}} \hspace{-2em}A(a_{k-1}, a_k) \times \cdots \times A(a_0, a_1)$}}="r";
   (0,-20)*+{1}="b";
       {\ar "t";"l"};
       {\ar "t";"r"};
       {\ar "r";"b"};
       {\ar "l";"b"};
(0,15)*{\dcorner};
\endxy \]
but for each $k \geq 0$ this pullback is just a product, so we have
\[(\Sigma P)(A)(a,a') = \displaystyle\coprod_{\begin{array}{c} {\scriptstyle k\geq 0} \\ {\scriptstyle a=a_0, \ldots, a_k=a' \in A(0)} \end{array}} \hspace{-2em}
P(k) \times  A(a_{k-1},a_k) \times \cdots \times A(a_0,a_1).\]
Hence an algebra action is given by, for all $k \geq 0,\ a_0, \ldots , a_k \in A_0$ a morphism
\[P(k) \times A(a_{k-1}, a_k) \times \cdots \times A(a_0,a_1) \lra A(a_0,a_k)\]
satisfying the relevant axioms to give precisely a $(\cl{V},P)$-category. This extends to an equivalence of categories straightforwardly.
\end{prf}

%

\begin{mydefinition}
We write $\cat{fc}_{(\cl{V},P)}$ for the cartesian monad on \cat{\cl{V}-Gph} associated to the operad $\Sigma P$.  By Proposition~\ref{sigmap}, its category of algebras is \cat{$(\cl{V},P)$-Cat}, so it is the monad for ``free $(\cl{V},P)$-categories''.
\end{mydefinition}

\begin{myexample}\label{twopointtwelve}
In the case $P=1$, we have $\cat{fc}_{(\cl{V},P)} = \cat{fc}_{\cl{V}}$.
\end{myexample}

This example shows how our construction is a generalisation of the free \mbox{\cl{V}-category} construction.  In \cite[Appendix F]{lei8} Leinster makes great use of the $\cat{fc}_{\cl{V}}$ construction to define the free strict $n$-category monads; in Section~\ref{four} we will proceed analogously using the $\cat{fc}_{(\cl{V},P)}$ construction to define the free weak $n$-category monads associated to an iterative operadic theory.

\section{Batanin's definition}\label{three}

We begin with some preliminaries on globular sets and pasting diagrams, which will also be useful in understanding the constructions of Section~\ref{four}.  We are going to need the notion of ``globular operad''; these operads are the same as $(\cat{GSet},T)$-operads where $T$ is the free strict $\omega$-category monad, which constructs pasting diagrams freely.  So it is useful to introduce some notation and terminology for pasting diagrams first.

\subsection{Globular sets and pasting diagrams}\label{threepointone}

\begin{mydefinition}
A \demph{globular set} $A$ is a diagram in \cat{Set}
\[\xy
(120,0)*+{A_0}="0";
(103,0)*+{A_1}="1";
(86,0)*+{A_2}="2";
(72,0)*+{}="3";
(62,0)*+{}="4";
(45,0)*+{A_{n-1}}="5";
(24,0)*+{A_{n}}="6";
(10,0)*+{}="7";
(67,0)*{\cdots};
(6,0)*{\cdots};
{\ar@<0.7ex>^{s} "1"; "0"};
{\ar@<-0.7ex>_{t} "1"; "0"};
{\ar@<0.7ex>^{s} "2"; "1"};
{\ar@<-0.7ex>_{t} "2"; "1"};
{\ar@<0.7ex>^<<<<<{s} "3"; "2"};
{\ar@<-0.7ex>_<<<<<{t} "3"; "2"};
{\ar@<0.7ex>^>>>>>>{s} "5"; "4"};
{\ar@<-0.7ex>_>>>>>>{t} "5"; "4"};
{\ar@<0.7ex>^<<<<<{s} "6"; "5"};
{\ar@<-0.7ex>_<<<<<{t} "6"; "5"};
{\ar@<0.7ex>^<<<<<<{s} "7"; "6"};
{\ar@<-0.7ex>_<<<<<<{t} "7"; "6"};
\endxy\]
satisfying the globularity conditions $ss=st$ and $ts=tt$.  Globular sets together with the obvious morphisms form a category \cat{GSet}, which can of course also be expressed as a presheaf category; we write \bb{G} for the ``globe'' category on which globular sets are presheaves, whose objects are the natural numbers.   In particular \cat{GSet} is cartesian.  An \demph{$n$-dimensional globular set} or simply \mbox{\demph{$n$-globular set}} is a globular set with $A_k=\emptyset$ for all $k>n$.  Similarly the $n$-globular sets form a category \cat{$n$-GSet}.
\end{mydefinition}

We will be making use of the fact that $n$-globular sets can be formed by iterating the $\cl{V}$-graph construction.  The following is proved as part of Proposition 1.4.9 in \cite{lei8}.

\begin{lemma}\label{lemmathreepointtwo}
Write $\cat{0-Gph}=\cat{Set}$ and for all $n>0$ put
\[\cat{$n$-Gph} = \cat{($(n-1)$-Gph)-Gph}.\]
Then for all $n \geq 0$ we have
\[\cat{$n$-Gph} \catequiv \cat{$n$-GSet}.\]
\end{lemma}

We will use \cat{$n$-Gph} and \cat{$n$-GSet} slightly interchangeably, tending to prefer the former when we are iterating an enrichment construction, and the latter when dealing with globular operads.

We now present the free strict $\omega$-category monad, which is treated in great detail in \cite{lei8}.  There is a forgetful functor
\[U:\cat{$\omega$-Cat} \lra \cat{GSet}\]
where \cat{$\omega$-Cat} is the category of strict $\omega$-categories.  The following is proved in \cite[Appendix F]{lei8}.

\begin{theorem}

The forgetful functor $U$ is monadic and the induced monad $T$ on \cat{GSet} is cartesian.  Likewise for the $n$-dimensional case, which induces a monad $T^{(n)}$ on \cat{$n$-GSet}.
\end{theorem}

Thus $T$ is the free strict $\omega$-category monad, and its action on a globular set $A$ produces all formal composites of cells in $A$.

\begin{mydefinition}
An \demph{$m$-pasting diagram} is an $m$-cell of $T1$, where 1 is the terminal globular set which has precisely one cell of each dimension.
\end{mydefinition}

\begin{myexample}\label{pdex}
We will use the following running example of a 2-pasting diagram.

\vspace{-6pt}\[
       \def\objectstyle{\scriptstyle}
       
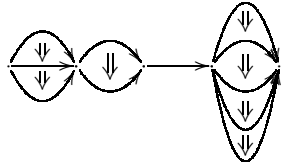
\xy
   (18,0)*{\cdot}="1";
   (27,0)*{\cdot}="2";
       {\ar@/^.8pc/ "1";"2"};
       {\ar@/_.8pc/ "1";"2"};
       {\ar@/^2pc/ "1";"2"};
       {\ar@/_2pc/ "1";"2"};
       {\ar@/_3pc/ "1";"2"};
       {\ar@{=>} (22.5,1.5)*{};(22.5,-1.5)*{}} ;
       {\ar@{=>} (22.5,7.25)*{};(22.5,4.75)*{}} ;
       {\ar@{=>} (22.5,-4.75)*{};(22.5,-7.25)*{}} ;
       {\ar@{=>} (22.5,-9.25)*{};(22.5,-11.75)*{}} ;
       (-9,0)*{\cdot}="1";
       (0,0)*{\cdot}="2";
           {\ar "1";"2"};
           {\ar@/^1.1pc/ "1";"2"};
           {\ar@/_1.1pc/ "1";"2"};
           {\ar@{=>} (-4.5,3)*{};(-4.5,.75)*{}} ;
           {\ar@{=>} (-4.5,-.75)*{};(-4.5,-3)*{}} ;
   (0,0)*{\cdot}="1";
   (9,0)*{\cdot}="2";
       {\ar@/^.8pc/ "1";"2"};
       {\ar@/_.8pc/ "1";"2"};
       {\ar@{=>} (4.5,1.75)*{};(4.5,-1.75)*{}} ;
   (9,0)*{\cdot}="1";
   (18,0)*{\cdot}="2";
    {\ar "1";"2"};
\endxy\]
\end{myexample}

\noindent In fact for precise calculations we will make use of Batanin's highly efficacious tree notation for pasting diagrams, which follows from the following immediate consequence of Leinster's construction of the free strict $\omega$-category monad.

\begin{proposition}
Let $m \geq 1$. An $m$-pasting diagram is given by a (possibly empty) sequence of $(m-1)$-pasting diagrams.
\end{proposition}

\begin{cor}
Let $m \geq 0$.  An $m$-pasting diagram is given by a (possibly degenerate) planar tree of height $m$.
\end{cor}

\begin{myexample}
The pasting diagram in Example~\ref{pdex} corresponds to the tree
%
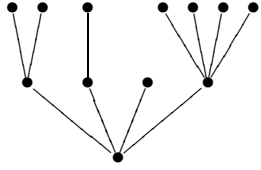
\[\xy
(-14,20)*{\bullet}="6";
(-10,20)*{\bullet}="7";
(-4,20)*{\bullet}="8";
(6,20)*{\bullet}="9";
(10,20)*{\bullet}="10";
(14,20)*{\bullet}="11";
(18,20)*{\bullet}="12";
(-12,10)*{\bullet}="2";
(-4,10)*{\bullet}="3";
(4,10)*{\bullet}="4";
(12,10)*{\bullet}="5";
(0,0)*{\bullet}="1";
"1";"2" **\dir{-};
"1";"3" **\dir{-};
"1";"4" **\dir{-};
"1";"5" **\dir{-};
"2";"6" **\dir{-};
"2";"7" **\dir{-};
"3";"8" **\dir{-};
"5";"9" **\dir{-};
"5";"10" **\dir{-};
"5";"11" **\dir{-};
"5";"12" **\dir{-};
\endxy\]
Note that some care is required for degenerate cases; the above picture also corresponds to a degenerate $m$-pasting diagram for each $m>2$.
\end{myexample}

\begin{mydefinition}
Given a tree $\tau$ its \demph{dimension} is the dimension of the corresponding pasting diagram, and we write $\cat{dim}(\tau)$; taking care to remember that for degenerate trees this will not be the height of the tree as drawn on the page.  If $\cat{dim}(\tau) = m$ then $\tau$ is called an \demph{$m$-stage tree}.
\end{mydefinition}

Every tree can be decomposed as a sequence of trees, and this corresponds to the expression of the corresponding $m$-pasting diagram as a sequence of $(m-1)$-pasting diagrams.  We will often give constructions or proofs by induction over dimension, expressing an $m$-stage tree $\alpha$ as a sequence $(\alpha_1, \ldots, \alpha_k)$ of $(m-1)$-stage trees, joined by a new base node as follows:

%
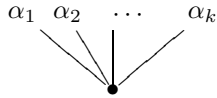
\[\xy
(-12,10)*+{\alpha_1}="2";
(-6,10)*+{\alpha_2}="3";
(0,10)*+{\phantom{\alpha_3}}="4";
(2,10)*+{\cdots};
(12,10)*+{\alpha_k}="5";
(0,0)*{\bullet}="1";
"1";"2" **\dir{-};
"1";"3" **\dir{-};
"1";"4" **\dir{-};
"1";"5" **\dir{-};
\endxy\]

\begin{myexample}
The above 2-stage tree is given as the sequence
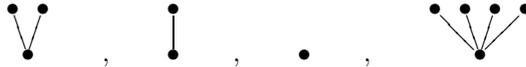
\[
\xy
(0,0)*{\bullet}="1";
(-2,6)*{\bullet}="2";
(2,6)*{\bullet}="3";
"1";"2" **\dir{-};
"1";"3" **\dir{-};
\endxy\hspace{2em},\hspace{2em}
\xy
(0,0)*{\bullet}="1";
(0,6)*{\bullet}="2";
"1";"2" **\dir{-};
\endxy\hspace{2em},\hspace{2em}
\xy
(0,0)*{\bullet}="1";
\endxy\hspace{2em},\hspace{2em}
\xy
(0,0)*{\bullet}="1";
(-6,6)*{\bullet}="2";
(-2,6)*{\bullet}="3";
(2,6)*{\bullet}="4";
(6,6)*{\bullet}="5";
"1";"2" **\dir{-};
"1";"3" **\dir{-};
"1";"4" **\dir{-};
"1";"5" **\dir{-};
\endxy\]
of 1-stage trees, where the third one is degenerate.  As a 2-pasting diagram, this corresponds to the sequence
\[\xy
(0,0)*{\cdot}="1";
(8,0)*{\cdot}="2";
(16,0)*{\cdot}="3";
{\ar "1";"2"};
{\ar "2";"3"};
\endxy\hspace{1.5em},\hspace{1.5em}
\xy
(0,0)*{\cdot}="1";
(8,0)*{\cdot}="2";
{\ar "1";"2"};
\endxy\hspace{1.5em},\hspace{1.5em}
\xy
(0,0)*{\cdot}="1";
\endxy\hspace{1.5em},\hspace{1.5em}
\xy
(0,0)*{\cdot}="1";
(8,0)*{\cdot}="2";
(16,0)*{\cdot}="3";
(24,0)*{\cdot}="4";
(32,0)*{\cdot}="5";
{\ar "1";"2"};
{\ar "2";"3"};
{\ar "3";"4"};
{\ar "4";"5"};
\endxy\]
of 1-pasting diagrams.  These in turn correspond to the four horizontally composed components of the pasting diagram
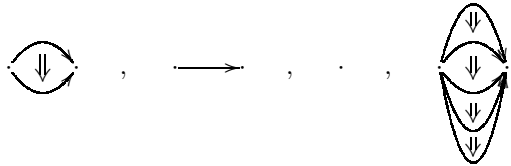
\[\xy
(0,0)*{\cdot}="1";
   (9,0)*{\cdot}="2";
       {\ar@/^.8pc/ "1";"2"};
       {\ar@/_.8pc/ "1";"2"};
       {\ar@{=>} (4.5,1.75)*{};(4.5,-1.75)*{}} ;
   \endxy\hspace{1.5em},\hspace{1.5em}
\xy
(9,0)*{\cdot}="1";
   (18,0)*{\cdot}="2";
    {\ar "1";"2"};
\endxy\hspace{1.5em},\hspace{1.5em}
\xy
(0,0)*{\cdot}="1";
\endxy\hspace{1.5em},\hspace{1.5em}
\xy
(18,0)*{\cdot}="1";
   (27,0)*{\cdot}="2";
       {\ar@/^.8pc/ "1";"2"};
       {\ar@/_.8pc/ "1";"2"};
       {\ar@/^2pc/ "1";"2"};
       {\ar@/_2pc/ "1";"2"};
       {\ar@/_3pc/ "1";"2"};
       {\ar@{=>} (22.5,1.5)*{};(22.5,-1.5)*{}} ;
       {\ar@{=>} (22.5,7.25)*{};(22.5,4.75)*{}} ;
       {\ar@{=>} (22.5,-4.75)*{};(22.5,-7.25)*{}} ;
       {\ar@{=>} (22.5,-9.25)*{};(22.5,-11.75)*{}} ;
       \endxy
\]
\end{myexample}

Evidently, in order to turn the above sequence of 1-pasting diagrams into the 2-pasting diagram, we have to shift the dimensions by 1.  This notion of ``suspension'' appears again in Section~\ref{five}.

\begin{myremark}\label{nodeheight}\label{remone}
The bottom node of a tree tells us the arity of the 0-composition involved in the pasting diagram; more generally the nodes at height $b$ tell us the arity of each string of $b$-composites in the pasting diagram.
\end{myremark}

\subsection{Batanin's definition} \label{threepointtwo}

We are now ready to give Batanin's definition of $n$-category.
Note that we are using a non-algebraic version of Leinster's variant of Batanin's definition, as discussed in the Introduction.

\begin{mydefinition}
A \demph{globular operad} is a $(\cat{GSet}, T)$-operad.  An $n$-dimensional globular operad is an $(\cat{$n$-GSet}, T^{(n)})$-operad.
\end{mydefinition}

Thus we have an underlying collection
\[\xy
{\ar^{d} (0,20)*+{P};(0,10)*+{T1}};
\endxy\]
equipped with unit and multiplication maps; $P$ can be thought of as the globular set of operations, and each operation has a pasting diagram as its arity, given by the map $d$.  The remaining structure we need for our definition of $n$-categories is the notion of ``contractibility''.

\begin{mydefinition}
Let $T$ be the free strict $\omega$-category monad on \cat{GSet}. A $T$-collection
\[\xy
{\ar^{p} (0,20)*+{A};(0,10)*+{T1}};
\endxy\]
is \demph{contractible} if

\begin{itemize}

\item given any 0-cells $a,b \in A$ and a 1-cell $y:pa \lra pb \in T1$, there exists a 1-cell $x:a \lra b \in A$ such that $px = y$, and

\item for all $m \geq 1$, given any $m$-cells $a,b \in A$ that are ``parallel'' i.e. $sa = sb$ and $ta=tb$, and an $(m+1)$-cell $y:pa \lra pb \in T1$, there exists an $(m+1)$-cell $x:a \lra b \in A$ such that $px = y$.

\end{itemize}
\end{mydefinition}

Note that this can be expressed as a lifting condition just as for fibrations in spaces or simplicial sets -- by the inclusion of the $m$-sphere into the boundary of the $(m+1)$-ball.

\begin{mydefinition}
For each $m \geq 0$ the \demph{globular $m$-sphere} $S^m$ is given by the following $m$-globular set
\[\xy
(90,0)*+{2}="0";
(75,0)*+{2}="1";
(60,0)*+{2}="2";
(45,0)*+{}="3";
(35,0)*+{}="4";
(20,0)*+{2}="5";
(5,0)*+{2}="6";
(40,0)*{\cdots};
{\ar@<0.7ex>^{0} "1"; "0"};
{\ar@<-0.7ex>_{1} "1"; "0"};
{\ar@<0.7ex>^{0} "2"; "1"};
{\ar@<-0.7ex>_{1} "2"; "1"};
{\ar@<0.7ex>^{0} "3"; "2"};
{\ar@<-0.7ex>_{1} "3"; "2"};
{\ar@<0.7ex>^{0} "5"; "4"};
{\ar@<-0.7ex>_{1} "5"; "4"};
{\ar@<0.7ex>^{0} "6"; "5"};
{\ar@<-0.7ex>_{1} "6"; "5"};
\endxy\]
where 2 denotes the 2-element set $\{0,1\}$ and the maps 0 and 1 send everything to 0 and 1 respectively.

The \demph{globular $m$-ball} $B^m$ is given by the following $m$-globular set
\[\xy
(90,0)*+{2}="0";
(75,0)*+{2}="1";
(60,0)*+{2}="2";
(45,0)*+{}="3";
(35,0)*+{}="4";
(20,0)*+{2}="5";
(5,0)*+{1}="6";
(40,0)*{\cdots};
{\ar@<0.7ex>^{0} "1"; "0"};
{\ar@<-0.7ex>_{1} "1"; "0"};
{\ar@<0.7ex>^{0} "2"; "1"};
{\ar@<-0.7ex>_{1} "2"; "1"};
{\ar@<0.7ex>^{0} "3"; "2"};
{\ar@<-0.7ex>_{1} "3"; "2"};
{\ar@<0.7ex>^{0} "5"; "4"};
{\ar@<-0.7ex>_{1} "5"; "4"};
{\ar@<0.7ex>^{0} "6"; "5"};
{\ar@<-0.7ex>_{1} "6"; "5"};
\endxy.\]
\end{mydefinition}
Thus there is an obvious inclusion $S^m \hookrightarrow B^{m+1}$, and a $T$-collection
\[\xy
{\ar^{p} (0,20)*+{A};(0,10)*+{T1}};
\endxy\]
is contractible if and only if for all $m \geq 0$ every commutative square

\[\xy
(0,20)*+++{\makebox[0.8em][l]{$S^m$}}="1";
(20,20)*+{A}="2";
(0,0)*+{B^{m+1}}="3";
(20,0)*+{T1}="4";
{\ar^{} "1";"2"};
{\ar^{} "2";"4"};
{\ar@{^{(}->}_{} "1";"3"};
{\ar_{} "3";"4"};
\endxy\]
has a lift
\[\xy
(0,20)*+++{\makebox[0.8em][l]{$S^m$}}="1";
(20,20)*+{A}="2";
(0,0)*+{B^{m+1}}="3";
(20,0)*+{T1}="4";
{\ar^{} "1";"2"};
{\ar^{} "2";"4"};
{\ar@{^{(}->}_{} "1";"3"};
{\ar_{} "3";"4"};
{\ar@{-->}_{} "3";"2"};
\endxy\]
making both triangles commute.  For the $n$-dimensional case we must ``collapse'' everything above the $n$th dimension.  Thus we have the following $n$-dimensional definition.

\begin{mydefinition}
Let $n \geq 0$ and $T^{(n)}$ the free strict $n$-category monad on \cat{$n$-GSet}.  A $T^{(n)}$-collection
\[\xy
{\ar^{p} (0,20)*+{A};(0,10)*+{T^{(n)}1}};
\endxy\]
is \demph{contractible} if

\begin{itemize}

\item given any 0-cells $a,b \in A$ and a 1-cell $y:pa \lra pb \in T1$, there exists a 1-cell $x:a \lra b \in A$ such that $px = y$,

\item for all $1 \leq m < n$, given any $m$-cells $a,b \in A$ that are ``parallel'' i.e. $sa = sb$ and $ta=tb$, and an $(m+1)$-cell $y:pa \lra pb \in T1$, there exists an $(m+1)$-cell $x:a \lra b \in A$ such that $px = y$.

\item given any parallel $n$-cells $a,b \in A$ and an $(n+1)$-cell $y:pa \lra pb \in T1$, we have $a = b$.

\end{itemize}
\end{mydefinition}

\begin{mydefinition}
A \demph{contractible globular operad} is a globular operad whose underlying $T$-collection is contractible.  Likewise for the $n$-dimensional case.
\end{mydefinition}

\begin{mydefinition}
A \demph{globular weak $\omega$-category} is an algebra for any contractible globular operad.  Likewise for the $n$-dimensional case.  \end{mydefinition}

\section{Comparison}\label{comparison}\label{four}

In this section we show how to start with an iterative operadic theory of $n$-categories and construct for each $n \geq 0$ an $n$-dimensional contractible globular operad whose algebras are the $n$-categories we started with.  The converse will in general not be possible since an iterative operadic theory of $n$-categories always has strict interchange, whereas a Batanin-style globular theory does not.

The data for an iterative operadic theory of $n$-categories is essentially just a series of operads $P_i$.  For each $n\geq 1$ we combine the operads $P_0, P_1, \ldots, P_{n-1}$ to produce an $n$-dimensional globular operad $Q^{(n)}$.  We then show that it is contractible whenever the $P_i$ are all contractible.  We will use the characterisation of a globular operad as a cartesian monad $Q$ equipped with a cartesian natural transformation $Q \Rightarrow T$.

We will follow the method used by Leinster in \cite[Appendix F]{lei8} to construct the free strict $n$-category monad $T^{(n)}$ by induction.  Leinster's approach uses three essential ingredients:

\begin{enumerate}
\item the free $\cl{V}$-category monad $\cat{fc}_{\cl{V}}$, for suitable categories \cl{V},
\item the ``suspension'' 2-functor
\[\cat{CAT} \lra \cat{CAT}\]
which sends $\cl{V}$ to \cat{\cl{V}-Gph}, and $F: \cl{V} \lra \cl{W}$ to $F_*$ which is the identity on objects and $F$ on hom-objects, preserving monadic adjunctions, cartesian monads and distributive laws, and
\item for suitable monads $T$ on \cl{V} a distributive law
\[ T_* \circ \cat{fc}_\cl{V} \Rightarrow \cat{fc}_\cl{V} \circ T_*\]
whose resulting composite monad $\cat{fc}_\cl{V} \circ T_*$ is the free \cat{$\cl{V}^T$-Cat} monad \cite[Proposition F.1.1]{lei8}.
\end{enumerate}

\noindent Then the monad $T^{(n)}$ is produced by induction, by putting $\cl{V} = \cat{$(n-1)$-Gph}$ and $T= T^{(n-1)}$.  Thus $\cl{V}^T = \bncat{n-1}$ and the distributive law
\[ T_* \circ \cat{fc}_\cl{V} \Rightarrow \cat{fc}_\cl{V} \circ T_*\]
gives us a monad whose category of algebras is $\cat{\bncat{n-1}-Cat}=\ncat{n}$, so we have constructed $T^{(n)}$ as required.

In order to construct the monad $Q^{(n)}$ for weak $n$-categories, we will generalise the above method to include the action of the operads $P_i$.  Afterwards in Section~\ref{fourpointtwo} we will describe how to construct this monad by a distributive series of monads; this is not necessary for the proofs, but sheds some light on the situation and provides a comparison with the analysis of the free strict $n$-category monad in \cite{che17}.

\subsection{Construction of a globular operad}\label{fourpointone}

We are going to copy the free $n$-category constructions in \cite[Appendix F]{lei8} but with composition parametrised by a series of operads $P_i$.  Thus, we need to replace the free $\cl{V}$-category functor in the construction with the free $(\cl{V},P)$-category functor, for suitable \cl{V} and $P$.  We then have three ingredients analogous to those described above and can make the analogous constructions, with some modification whenever the action of the $P_i$ is involved.  We therefore need the following three essential ingredients:

\begin{enumerate}
\item the free $(\cl{V},P)$-category monad $\cat{fc}_{(\cl{V},P)}$, for operads $P$ in suitable categories \cl{V},
\item the suspension functor as above, and
\item a certain distributive law, generalised from the one above.
\end{enumerate}

We have already constructed the monad $\cat{fc}_{(\cl{V},P)}$ in Section~\ref{twopointtwo}.  The distributive law is given by the following result, which generalises the distributive law provided in the proof of Proposition F.1.1 in \cite{lei8}.

\begin{theorem}\label{keynew}\label{theoremfourpointone}
\label{thmfourpointone}
Let \cl{V} be a presheaf category, $T$ a monad on \cl{V} that preserves coproducts, and $P$ an operad in $\cl{V}^T$, the category of algebras of $T$.  Write $U$ for the forgetful functor $\cl{V}^T \lra \cl{V}$, and $UP$ for the operad in $\cl{V}$ given by $(UP)(k) = U(P(k))$. Then we have monads on \cat{$\cl{V}$-Gph} given by $T_*$ and $\cat{fc}_{(\cl{V},UP)}$, and a distributive law
\[ \lambda: T_* \circ \cat{fc}_{(\cl{V},UP)} \Rightarrow \cat{fc}_{(\cl{V},UP)} \circ T_*\]
whose resulting composite monad $\cat{fc}_{(\cl{V},UP)} \circ T_*$ is the free \cat{$(\cl{V}^T,P)$-Cat} monad, that is
\[ (\cat{\cl{V}-Gph})^{\cat{{\small fc}}_{(\cl{V},UP)} \circ T_*} \cong \cat{$(\cl{V}^T,P)$-Cat}.\]

\end{theorem}

\begin{proof}
Since $T$ preserves coproducts we have
\[ (T_* \circ \cat{fc}_{(\cl{V},UP)})  (A)(a,a') = \coprod_{\begin{array}{c} {\scriptstyle k\geq 0} \\ {\scriptstyle a=a_0, \ldots, a_k=a' \in A(0)} \end{array}} \hspace{-3em}T\big(\ \ U(P(k) \times A(a_{k-1}, a_k) \times \cdots \times A(a_0, a_1)\ \big)\]
and also we have
\[ (\cat{fc}_{(\cl{V},UP)} \circ T_*) (A)(a,a') = \coprod_{\begin{array}{c} {\scriptstyle k\geq 0} \\ {\scriptstyle a=a_0, \ldots, a_k=a' \in A(0)} \end{array}} \hspace{-3em}U(P(k)) \times T\big(A(a_{k-1}, a_k)\big) \times \cdots \times T\big(A(a_0, a_1)\big).\]
Now the universal property of the product
\[U(P(k)) \times T\big(A(a_{k-1}, a_k)\big) \times \cdots \times T\big(A(a_0, a_1)\big)\]
together with the algebra action
\[TU(P(k)) \lra U(P(k))\]
induce a canonical morphism
\[\xy
(0,13)*+{T\big(\ (U(P(k)) \times A(a_{k-1}, a_k) \times \cdots \times A(a_0, a_1)\ \big)}="1";
(0,0)*+{U(P(k)) \times T\big(A(a_{k-1}, a_k)\big) \times \cdots \times T\big(A(a_0, a_1)\big)}="2";
{\ar^{} "1";"2"};
\endxy\]
and this gives us the components of a natural transformation $\lambda$ as required.  It is straightforward to check that $\lambda$ is a distributive law and that we have the isomorphism of categories required. \end{proof}

Note that this distributive law will correspond to the interchange laws in our weak $n$-categories; we discuss this further at the end of Section~\ref{fourpointtwo}.

\begin{myexample}
If $P=1$ then $\cat{fc}_{(\cl{V}, UP)} = \cat{fc}_{\cl{V}}$ and the result becomes exactly the result of Leinster.
\end{myexample}

We can now use this theorem to construct the monads $Q^{(n)}$ for weak $n$-categories.  Note that \emph{a priori} we have monads on \cat{$\cl{V}_{n-1}$-Gph}, but we seek to construct monads on \cat{$n$-GSet}.

\begin{theorem}
Let $P_0, P_1, \ldots, P_{i}, \ldots$ give an iterative operadic theory of $n$-categories, with a resulting category $\cl{V}_n$ of $n$-categories for each $n \geq 0$.  Then for each $n \geq 0$ there is a monad $Q^{(n)}$ on \cat{$n$-GSet} whose category of algebras is $\cl{V}_n$.
\end{theorem}

\begin{proof}
By induction.  First set $Q^{(0)} = \cat{id}$ as a monad on $\cl{V}_0 = \cat{Set}$.

Now let $n>0$.  We use Theorem~\ref{keynew} with
\begin{itemize}
\item $\cl{V} = \cat{$(n-1)$-GSet}$,
\item $T=Q^{(n-1)}$ so $\cl{V}^T = \cl{V}_{n-1}$ and $U$ is the forgetful functor
    \[\cl{V}_{n-1} \lra \cat{$(n-1)$-GSet},\]
\item $P=P_{n-1}$, which is an operad in $\cl{V}^T = \cl{V}_{n-1}$ as required,
\end{itemize}
and we set
\[Q^{(n)} = \cat{fc}_{(\cl{V}, UP)} \circ (T)_*\]

Then assuming $T$ preserves coproducts, we know by Theorem~\ref{keynew} that we have a distributive law
\[ \lambda: T_* \circ \cat{fc}_{(\cl{V},UP)} \Rightarrow \cat{fc}_{(\cl{V},UP)} \circ T_*\]
and the induced composite monad $\cat{fc}_{(\cl{V},UP)} \circ T_*$ has category of algebras
\[\begin{array}{rcl}
 \cat{$(\cl{V}^T,P)$-Cat} &=& \cat{$(\cl{V}_{n-1}, P_{n-1})$-Cat} \\
 &=& \cl{V}_n
 \end{array}\]
as required.  To make the induction go through, it remains to check that the induced monad preserves coproducts.  Now a routine calculation shows that the monad $\cat{fc}_{(\cl{V}, UP)}$ preserves coproducts, since products and coproducts commute in \cat{$(n-1)$-GSet}; also $T_*$ preserves coproducts if $T$ does, so by induction the result follows.
\end{proof}

We can now show almost immediately that this gives a globular operad; we use the characterisation of a globular operad as a cartesian monad $Q$ on \cat{GSet} equipped with a cartesian natural transformation $Q \Rightarrow T$.

\begin{theorem}
For each $n>0$ the monad $Q^{(n)}$ as above is cartesian, and there is a cartesian natural transformation $Q^{(n)} \Rightarrow T^{(n)}$.  Thus we have a globular operad whose category of algebras is $\cl{V}_n$.
\end{theorem}

\begin{proof}
By induction.  Clearly $Q^{(0)}$ is cartesian.  Now let $n>0$.  We know that $\cat{fc}_{(\cl{V},UP)}$ is cartesian as it is the monad associated to an operad (see Theorem~\ref{operadmonad}), and it is a straightforward exercise to check that if monads $S$ and $T$ are cartesian, then $ST$ is cartesian.  Thus $Q^{(n)} = \cat{fc}_{(\cl{V},UP)} \circ (Q^{(n-1)})_*$ is cartesian.

Now we need to exhibit a cartesian natural transformation $Q^{(n)} \Rightarrow T^{(n)}$.  Again we proceed by induction.  $Q^{(0)} = T^{(0)}$, so let $n>0$.  We know that
\[\begin{array}{rcl}
T^{(n)} &=& \cat{fc}_{\cl{V}} \circ (T^{(n-1)})_*\\
&=& \cat{fc}_{(\cl{V},1)} \circ (T^{(n-1)})_*
\end{array}\]
where here 1 is the terminal operad.  So we need a natural transformation
\[\cat{fc}_{(\cl{V},UP)} \circ (Q^{(n-1)})_* \Rightarrow \cat{fc}_{(\cl{V},1)} \circ (T^{(n-1)})_*.\]
Now, since $1$ is terminal, we have an operad map $UP \lra 1$ which induces a cartesian natural transformation $\cat{fc}_{(\cl{V}, UP)} \Rightarrow \cat{fc}_{(\cl{V}, 1)}$.  Composing this with the cartesian natural transformation $Q^{(n-1)} \Rightarrow T^{(n-1)}$ gives the natural transformation as required.
\end{proof}

This result gives us the first part of the comparison theorem which is the main result of this work: we have constructed, for any iterative operadic theory of $n$-categories, a cartesian natural transformation $Q^{(n)} \Rightarrow T^{(n)}$ and thus a globular operad.  It remains to show that the resulting operad is contractible; this depends on the contractibility of each of the operads $P_i$.  In order to prove this we will need to use the explicit description of the globular operad associated to the $P_i$.

We need to give, for each pasting diagram $\alpha \in T1$ the set of operations of $Q^{(n)}1$ lying over $\alpha$; we will write this as $Q^{(n)}_\alpha$.  We will also give the globular source and target maps in $Q^{(n)}1$ as these are important for examining contractibility.  We will use the tree notation for pasting diagrams.  Recall that an $m$-pasting diagram is described by a tree of dimension $m$.  For example, the following is a 4-stage tree:
%
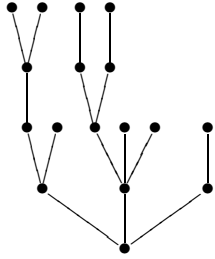
\[\xy
(5,32)*{\bullet}="14";
(9,32)*{\bullet}="15";
(14,32)*{\bullet}="16";
(18,32)*{\bullet}="17";
(7,24)*{\bullet}="11";
(14,24)*{\bullet}="12";
(18,24)*{\bullet}="13";
(7,16)*{\bullet}="5";
(11,16)*{\bullet}="6";
(16,16)*{\bullet}="7";
(20,16)*{\bullet}="8";
(24,16)*{\bullet}="9";
(31,16)*{\bullet}="10";
(9,8)*{\bullet}="2";
(20,8)*{\bullet}="3";
(31,8)*{\bullet}="4";
(20,0)*{\bullet}="1";
"1";"2" **\dir{-};
"1";"3" **\dir{-};
"1";"4" **\dir{-};
"2";"5" **\dir{-};
"2";"6" **\dir{-};
"3";"7" **\dir{-};
"3";"8" **\dir{-};
"3";"9" **\dir{-};
"4";"10" **\dir{-};
"5";"11" **\dir{-};
"7";"12" **\dir{-};
"7";"13" **\dir{-};
"11";"14" **\dir{-};
"11";"15" **\dir{-};
"12";"16" **\dir{-};
"13";"17" **\dir{-};
\endxy\]
Recall that the nodes of height $b$ describe pasting along bounding $b$-cells.  The globular source and target in $T1$ are found by forgetting the top level of the tree.  In the above example we get the following 3-stage tree as both source and target:
%
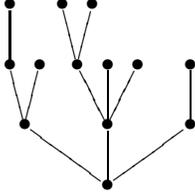
\[\xy
(7,24)*{\bullet}="11";
(14,24)*{\bullet}="12";
(18,24)*{\bullet}="13";
(7,16)*{\bullet}="5";
(11,16)*{\bullet}="6";
(16,16)*{\bullet}="7";
(20,16)*{\bullet}="8";
(24,16)*{\bullet}="9";
(31,16)*{\bullet}="10";
(9,8)*{\bullet}="2";
(20,8)*{\bullet}="3";
(31,8)*{\bullet}="4";
(20,0)*{\bullet}="1";
"1";"2" **\dir{-};
"1";"3" **\dir{-};
"1";"4" **\dir{-};
"2";"5" **\dir{-};
"2";"6" **\dir{-};
"3";"7" **\dir{-};
"3";"8" **\dir{-};
"3";"9" **\dir{-};
"4";"10" **\dir{-};
"5";"11" **\dir{-};
"7";"12" **\dir{-};
"7";"13" **\dir{-};
\endxy\]

\begin{mydefinition}
Let $A \in \cl{V}_n$ and write $U$ for the forgetful functor $\cl{V}_n \lra \cat{$n$-GSet}$.  Then the \demph{$m$-cells of $A$} are simply the $m$-cells of $UA$.
\end{mydefinition}

\begin{proposition}\label{propfourpointfive}
\label{propositionfourpointfive}
For all $n > 0$, an element of $Q^{(n)}(1)$ lying over a tree $\alpha \in T^{(n)}(1)$ consists of a label for each node of the tree, where given $k \geq 0$ and $0 \leq b < m \leq n$, a node of arity $k$ at height $b$ in an $m$-stage tree must be an $(m-b-1)$-cell of $P_{n-b-1}(k)$.

The source (respectively target) of this element is found by
\begin{enumerate}
\item forgetting the top level of the tree and the top level of labels, and
\item replacing each remaining label with its source (respectively target).
\end{enumerate}
\end{proposition}

\begin{proof}
We proceed by induction over $n$.  For $n=1$ we are considering 1-stage trees, which are determined simply by an integer $k \geq 0$.   Now $Q^{(1)}$ is the monad induced by the adjunction
\[\xy
(40,0)*+{\cat{$1$-GSet}}="1";
(0,0)*+{\cat{$( \cat{Set}, P_0)$-Cat}}="2";
(22,0)*{{\scriptstyle \top}};
{\ar@<1ex>^<<<<<<<<<<<{\mbox{\textsf{forgetful}}} "2";"1"};
{\ar@<1ex>^<<<<<<<<<<<{\mbox{\textsf{free}}} "1";"2"};
\endxy.\]
This monad leaves 0-cells unchanged; 1-cells of $Q^{(1)}(1)$ are given by the set
\[\coprod_{k \geq 0} P_0(k)_0.\]
The canonical map to $T^{(1)}(1)$ is given by mapping elements of $P_0(k)_0$ to the arity $k$, as required.

Now let $n>1$.  We know that
\[Q^{(n)} = \cat{fc}_{(\cl{V}, UP)} \circ (Q^{(n-1)})_*\]
where $\cl{V} = \cat{$n$-GSet}$ and $P = P_{n-1}$.  So, using the definition of the monad $\cat{fc}_{(\cl{V},UP)}$ we see that $Q^{(n)}$ leaves 0-cells unchanged, and for $m >0$ the $m$-cells of $Q^{(n)}(1)$ are given by the following set
\[\big( Q^{(n)}(1)\big)_m = \big( \coprod_{k \geq 0} UP_{n-1}(k)_{m-1} \times \big( Q^{(n-1)}(1) \big)^k\big)_{m-1}.\]
Now by induction we are done, since for an $m$-stage tree $\alpha$ given by a sequence $(\alpha_1, \ldots, \alpha_k)$ of $(m-1)$-stage trees:
\begin{itemize}
\item for the bottom node of the tree ($b=0$) we get a label in $P_{n-1}(k)_{m-1}$ as required, where $k$ is the arity of the node, and
\item a node of height $b>0$ is a node of height $b-1$ in one of the $(m-1)$-stage subtrees $\alpha_i$ which, by the formula above, has labels given by $Q^{(n-1)}(1)$ giving us the dimensions we require.
\end{itemize}
\end{proof}

Recall (Remark~\ref{remone}) that the nodes of height $b$ tell us about $b$-composition, so by comparison with Table~\ref{bigtable} we can check informally that the labelling described above parametrises composition in the way we expected.

We are now ready to prove the contractibility result.

\begin{proposition}\label{propcoh}
With notation as before, given $n \geq 0$ the globular operad $Q^{(n)}$ is contractible if and only if for all $0 \leq i < n$ the operad $P_i$ is contractible.
\end{proposition}

\begin{proof}
We use the explicit construction above.  We write the underlying collection of the globular operad $Q^{(n)}$ as
\[\xy
{\ar^{p} (0,20)*+{Q};(0,10)*+{T1}};
\endxy\]
omitting the $n$ superscripts for convenience.  We consider parallel $m$-cells $a,b \in Q$ and an $(m+1)$-cell $\alpha: pa \lra pb \in T1$.
First note that the globular source and target maps in $T1$ are equal so the existence of $\alpha$ tells us $pa = pb$ as $m$-stage trees.  We write this tree as $\vardelta \alpha$; it is the $(m+1)$-stage tree $\alpha$ with the top level omitted; the $m$-cells $a$ and $b$ in $Q$ consist of labels for this tree.


Consider the case $m < n$.  We need to find an $(m+1)$-cell $x: a \lra b \in Q$ such that $px = \alpha$.  Thus, by Proposition~\ref{propfourpointfive}, we need to find a label $\alpha_\nu$ for each node $\nu$ of height $b \leq m$ of the tree $\alpha$ such that:
\begin{enumerate}
\item if $b < m$ then $s(\alpha_\nu)$ is the label for $\nu$ in $a$ and $t(\alpha_\nu)$ is the label for $\nu$ in $b$, and
\item if $b = m$ then $\alpha_\nu$ is 0-cell of $P_{n-m-2}$.
\end{enumerate}
For (1) the existence of such a cell follows from the contractibility of $P_{n-b-1}$, and for (2) the existence of such a cell comes from the non-empty condition in the contractibility of $P_{n-m-2}$.

For the case $m=n$ we need to show, under the above hypotheses, that $a=b$ i.e. both give the same labels for each node of $\vardelta \alpha$. In this case the label of a $k$-ary node of height $b$ is an $(n-b-1)$-cell of $P_{n-b-1}(k)$, a contractible \mbox{$(n-b-1)$-category}.  By contractibility there is only one such cell (see Remark~\ref{remtwo}), so we must have $a=b$.

%
%
%
%

The converse follows since if $Q$ is contractible we can use the above construction to find cells in between any parallel cells of $P_i(k)$ as required.
\end{proof}

We now sum up the results of this section in the main theorem as follows.

\begin{theorem}
Suppose we have operads $P_0, P_1, \ldots, P_{i}, \ldots$ giving an iterative operadic theory of $n$-categories, with a resulting category $\cl{V}_n$ of $n$-categories for each $n \geq0$.  Then for each $n \geq 0$ there is a contractible globular operad $Q^{(n)}$ such that
\[\cat{Alg}\ Q^{(n)} \catequiv \cl{V}_n.\]
\end{theorem}

\subsection{Iterated distributive laws for weak $n$-categories}\label{fourpointtwo}

In this section we will show how iterative operadic weak $n$-categories can be constructed using iterated distributive laws as in \cite{che17}.  In the previous section we began with an iterative operadic theory of weak $n$-categories, and constructed free weak $n$-category monads $Q^{(n)}$ on \cat{$n$-GSet} by induction, as
\[Q^{(n)} = \cat{fc}_{(\cl{V}, UP)} \circ (Q^{(n-1)})_*\]
where $\cl{V}=\cat{$(n-1)$-Gph}$, and $P=P_{n-1}$.  In this section we will further discuss what this actually gives us when the induction is unravelled.  We will show that we have monads $Q^{(n)}_i$ for parametrised $i$-composition, analogous to the monads $T^{(n)}_i$ for strict $i$-composition described in \cite{che17}.

\begin{lemma}
Let $n>0$ and $0 \leq i < n$.  Set
\[Q^{(n)}_i = (\cat{fc}_{(\cl{V},P)})_{i*}\]
as a monad on \cat{$n$-Gph}, where
\begin{itemize}
\item $\cl{V}= \cat{$(n-i-1)$-GSet}$,
\item $P=U_{n-i-1}P_{n-i-1}$ where we are writing $U_{n-i-1}$ for the forgetful functor $\cl{V}_{n-i-1} \lra \cat{$(n-i-1)$-GSet}$, and
\item $(-)_{i*}$ denotes applying $(-)_*$ $i$ times.
\end{itemize}
Then $Q^{(n)}=Q^{(n)}_0 Q^{(n)}_1 \cdots Q^{(n)}_{n-1}$.
\end{lemma}

\begin{proof}
Straightforward by induction.
\end{proof}

\begin{myremark}
An algebra for $Q^{(n)}_i$ is an $n$-globular sets with parametrised $i$-composition but no other composition.
\end{myremark}

In \cite{che17} we showed that the monads
\[T^{(n)}_0, T^{(n)}_1, \ldots, T^{(n)}_{n-1}\]
for strict $i$-composition form a ``distributive series of monads'', shedding light on the interchange laws for strict $n$-categories.  We now prove analogous results for the weak case; first we recall the relevant results from \cite{che17}.

\begin{thm}\label{mainthmdist}

Fix $n \geq 3$.  Let $T_1, \ldots, T_n$ be monads on a category \cl{C}, equipped with

\begin{itemize}

\item for all $i > j$ a distributive law $\lambda_{ij}: T_i T_j \Rightarrow T_j T_i$, satisfying

\item for all $i>j>k$ the ``Yang-Baxter'' equation given by the commutativity of the following diagram

\begin{equation}\label{yb}
\xy
(0,12)*+{T_i T_j T_k}="1";
(12,24)*+{T_j T_i T_k}="2";
(36,24)*+{T_j T_k T_i}="3";
(12,0)*+{T_i T_k T_j}="4";
(36,0)*+{T_k T_i T_j}="5";
(48,12)*+{T_k T_j T_i}="6";
{\ar^{\lambda_{ij} T_k} "1";"2"};
{\ar^{T_j \lambda_{ik}} "2";"3"};
{\ar^{\lambda_{jk} T_i} "3";"6"};
{\ar_{T_i\lambda_{jk}} "1";"4"};
{\ar_{\lambda_{ik} T_j} "4";"5"};
{\ar_{T_k \lambda_{ij}} "5";"6"};
\endxy\end{equation}

\end{itemize}

\noindent Then for all $1\leq i < n$  we have canonical monads
\[T_1 T_2 \cdots T_i \quad \mbox{and} \quad T_{i+1} T_{i+2} \cdots T_n\]

\noindent together with a distributive law of \ $T_{i+1} T_{i+2} \cdots T_n$\  over \ $T_1 T_2 \cdots T_i$\ i.e.
\[(T_{i+1} T_{i+2} \cdots T_n) (T_1 T_2 \cdots T_i) \Rightarrow (T_1 T_2 \cdots T_i)(T_{i+1} T_{i+2} \cdots T_n)\]

\noindent given by the obvious composites of the $\lambda_{ij}$.  Moreover, all the induced monad structures on \ $T_1 T_2 \cdots T_n$ are the same.

\end{thm}

\begin{mydefinition}
We refer to a series of monads as above as a \demph{distributive series of monads}.
\end{mydefinition}

\begin{thm}\label{fourpointseven}
Let $n \geq 0$. For all $0 \leq i < n$ there is a monad $T^{(n)}_i$ for free $i$-composition on \cat{$n$-GSet}.  For $n \geq 2$ and $n>i>j\geq 0$ there is a distributive law
\[\lambda^{(n)}_{ij}: T^{(n)}_i T^{(n)}_j \Rightarrow T^{(n)}_j T^{(n)}_i\]
given by the interchange law of $i$-composition and $j$-composition.
For $n \geq 3$ the monads $T^{(n)}_0, \cdots, T^{(n)}_{n-1}$ on \cat{$n$-GSet} form a distributive series of monads.  The resulting composite monad $T^{(n)}_0 T^{(n)}_1 \cdots T^{(n)}_{n-1}$ is the free strict $n$-category monad $T^{(n)}$ on \cat{$n$-GSet}.
\end{thm}

The analogous results hold for the weak case as follows.

\begin{theorem}
Let $n \geq 3$. Then the monads
\[Q^{(n)}_0, \cdots, Q^{(n)}_{n-1}\]
on \cat{$n$-GSet} form a distributive series of monads as in Theorem~\ref{mainthmdist}.
\end{theorem}

\begin{proof}
The distributive laws are given by instances of the distributive law of Theorem~\ref{keynew}.  It is straightforward to check that the Yang-Baxter equations hold.
\end{proof}

Here the distributive laws come from ``parametrised interchange laws'' -- since composition is parametrised by the action of the operads $P_i$, we must take the operads into account when describing interchange.

To shed some light on this we will give an example for the usual interchange of horizontal and vertical composition.  In a strict $n$-category, the usual interchange law says that given composable 2-cells
%
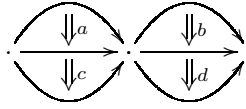
\[\xy
(30,0)*{ 	
\xy
(0,0)*{
\xy
(-3,0)*+{.}="1";
(13,0)*+{.}="2";
{\ar@/^1.55pc/^{} "1";"2"};
{\ar@/_1.55pc/_{} "1";"2"};
{\ar "1";"2"};
{\ar@{=>}^{a} (5,5)*{};(5,1)*{}};
{\ar@{=>}^{c} (5,-1)*{};(5,-5)*{}};
\endxy};
(16,0)*{       
\xy
(-3,0)*+{.}="1";
(13,0)*+{.}="2";
{\ar@/^1.55pc/^{} "1";"2"};
{\ar@/_1.55pc/_{} "1";"2"};
{\ar "1";"2"};
{\ar@{=>}^{b} (5,5)*{};(5,1)*{}};
{\ar@{=>}^{d} (5,-1)*{};(5,-5)*{}};
\endxy}
\endxy}
\endxy\]
we have
\[(d * c) \circ (b*a) = (d \circ b) * (c \circ a).\]
This corresponds to the distributive law
\[T^{(n)}_1 T^{(n)}_0 \Rightarrow T^{(n)}_0 T^{(n)}_1.\]
%
%
%
%

\noindent In the case of an iterative operadic $n$-category, we use the distributive law
\[Q^{(n)}_1 Q^{(n)}_0 \Rightarrow Q^{(n)}_0 Q^{(n)}_1.\]
Examining the formulae given in Proposition~\ref{propositionfourpointfive} we find that for our underlying data on the left hand side we now have:
\begin{itemize}
\item composable 2-cells $a,b,c,d$ as above,
\item for parametrising the two instances of 0-composition, composable 1-cells $f,g$ of $P_{n-1}(2)$, and
\item for parametrising the 1-composition, a 0-cell $\alpha$ of $P_{n-2}(2)$.
\end{itemize}
Then the parametrised interchange law says:
\[(d *_g c) \circ_\alpha (b*_f a) = (d \circ_\alpha b) *_{g \circ_\alpha f} (c \circ_\alpha a)\]
where $*_g$ denotes 0-composition parametrised by $g$; likewise $*_f$ and $\circ_\alpha$. Note that $g \circ_\alpha f$ is the result of composing $f$ and $g$ parametrised by $\alpha$ -- recall that $f$ and $g$ are 1-cells of $P_{n-1}(2) \in \cl{V}_{n-1}= \cat{$(\cl{V}_{n-2}, P_{n-2})$-Cat}$ and so $\alpha$ is a valid cell to parametrise the composition of 1-cells $f$ and $g$.  This part of the interchange law comes from the part of the proof of Theorem~\ref{keynew} involving the algebra action $TU(P(k)) \lra U(P(k))$.
Alternatively we could in general write $\alpha(g,f)$ instead of $g \circ_\alpha f$ in which case the parametrised distributive law becomes
\[\alpha\big(g(d,c),f(b,a)\big) = \alpha(g,f)\big(\alpha(d,b),\alpha(c,a)\big)\]
which emphasises the connection with operad composition, but leaves the connection with the original interchange law rather less obvious.

The general interchange law can be written down similarly.

\section{What Trimble's operad gives}\label{five}

In this section we apply our constructions to Trimble's original operad $E$ for $n$-categories, and conjecture that the resulting globular operad is a suboperad of the one Batanin uses to take fundamental $n$-groupoids of a space.  We give an operad morphism from one to the other whose properties we will study in a future work.

Again, we use the ``non-algebraic Leinster'' variant of Batanin's operad; this is the operad used by Cisinski in his work on Batanin fundamental $n$-groupoids \cite{cis1}.

\subsection{Batanin's operad}

To define Batanin's operad we need the geometric realisation of pasting diagrams, via associated globular sets.  Globular pasting diagrams arise among globular sets as those that are connected and loop-free. The globular set associated to a pasting diagram can also be constructed directly as in \cite[Section 8.1]{lei8} or by induction using the usual expression of an $m$-pasting diagram as a series of $(m-1)$-pasting diagrams, together with a ``suspension functor'' $\sigma$ that turns every $k$-cell into a $(k+1)$-cell, and adjoins two new 0-cells to be the source and target of all the other cells.

We will not go into the details here.  The idea is simply that, given a pasting diagram, we can form a globular set which consists of all the cells we have actually ``drawn'' in the pasting diagram, and we can then geometrically realise it.

\begin{mydefinition}
The \demph{geometric realisation} of a globular set is a functor
\[ |-| : \cat{GSet} = [\bb{G}^{\op}, \cat{Set}] \lra \cat{Top}\]
defined by Kan extension as follows.  We have a functor
\[ |-| : \bb{G} \lra \cat{Top} \]
defined by sending an object $m \in \bb{G}$ to the Euclidean $m$-disc $D^m$; cosource and cotarget maps are given by the lower and upper hemisphere inclusions and coidentity operators are given by orthogonal projection $D^{m+1} \lra D^m$.
We then take the left Kan extension
%
\[\xy
(0,25)*+{\bb{G}}="1";
(25,25)*+{[\bb{G}^{\op}, \cat{Set}]}="2";
(0,0)*+{\cat{Top}}="3";
{\ar^<<<<<<<{\cat{Yon}} "1";"2"};
{\ar_{|-|} "1";"3"};
{\ar@{-->}_{} "2";"3"};
\endxy\]
This is the geometric realisation functor for globular sets. Given a pasting diagram $\alpha$ we also write $|\alpha|$ for the geometric realisation of its associated pasting diagram.
\end{mydefinition}

We are now ready to describe for each $n \geq 0$ the operad $K^{(n)}$ used by Batanin to define fundamental $n$-groupoids of a space.  We first define the $\omega$-dimensional version $K$; see for example \cite[Proposition 9.2]{bat1} and \cite[Example 9.2.7]{lei8}.

\begin{mydefinition}
We write $K$ for the globular operad defined as follows.  Given an $m$-pasting diagram $\alpha$, the elements of $K_\alpha$ are the continuous maps $D^m \lra |\alpha|$ respecting the boundaries.  In order to say this more precisely, consider an $m$-dimensional pasting diagram $\alpha$, with source and target $\vardelta \alpha$. We have inclusion maps
\[\xy
(120,0)*+{\alpha}="0";
(100,0)*+{\partial\alpha}="1";
{\ar@<0.7ex>^{s} "1"; "0"};
{\ar@<-0.7ex>_{t} "1"; "0"};
\endxy\]
on the associated globular sets, and thus we have maps in \cat{Top}
\[\xy
(120,0)*+{|\alpha|}="0";
(100,0)*+{|\partial\alpha|}="1";
{\ar@<0.7ex>^{s} "1"; "0"};
{\ar@<-0.7ex>_{t} "1"; "0"};
\endxy.\]
To be ``boundary preserving'', our map $D^m \map{f} |\alpha|$ must have restrictions making the following diagrams commute
%
\[\xy
(0,20)*+{D^{m-1}}="1";
(20,20)*+{|\vardelta\alpha|}="2";
(0,0)*+{D^m}="3";
(20,0)*+{|\alpha|}="4";
{\ar@{..>}^{sf} "1";"2"};
{\ar^{s} "2";"4"};
{\ar_{s} "1";"3"};
{\ar_{f} "3";"4"};
\endxy
\hspace{4em}
\xy
(0,20)*+{D^{m-1}}="1";
(20,20)*+{|\vardelta\alpha|}="2";
(0,0)*+{D^m}="3";
(20,0)*+{|\alpha|}="4";
{\ar@{..>}^{tf} "1";"2"};
{\ar^{t} "2";"4"};
{\ar_{t} "1";"3"};
{\ar_{f} "3";"4"};
\endxy\]
Note that we do not have to have $sf = tf$; this is in fact crucial to allow for coherence maps between \emph{different composites} of cells.  In the language of Batanin, we have a map of ``coglobular spaces''.  Finally for the finite-dimensional case $K^{(n)}$ we take homotopy classes at dimension $n$.

Note that all these operads are contractible, so algebras for it are indeed $n$-categories.
\end{mydefinition}

\begin{myexample}
Continuing with the pasting diagram given in Example~\ref{pdex}, the boundary preserving maps $D^2 \lra |\alpha|$ are maps of spaces depicted below

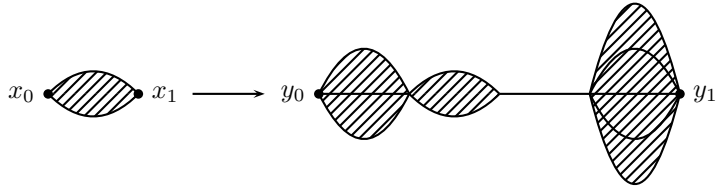
\psset{yunit=0.6cm,xunit=0.6cm,runit=0.6cm}
\pspicture(-7.5,-3)(0,3)
$\psdots(-6,0)(-4,0)
\parabola[fillstyle=hlines,hatchsep=2pt](-6,0)(-5,0.5)
\parabola[fillstyle=hlines,hatchsep=2pt](-6,0)(-5,-0.5)
\psline[arrows=->](-2.8,0)(-1.2,0)
\uput[180](-6,0){x_0}
\uput[0](-4,0){x_1}
\uput[180](0,0){y_0}
\uput[0](8,0){y_1}
\parabola[fillstyle=hlines,hatchsep=2pt](0,0)(1,1)
\parabola[fillstyle=hlines,hatchsep=2pt](0,0)(1,-1)
\parabola[fillstyle=hlines,hatchsep=2pt](2,0)(3,0.5)
\parabola[fillstyle=hlines,hatchsep=2pt](2,0)(3,-0.5)
\parabola[fillstyle=hlines,hatchsep=2pt](6,0)(7,2)
\parabola[fillstyle=hlines,hatchsep=2pt](6,0)(7,-2)
\parabola(6,0)(7,1)
\parabola(6,0)(7,-1)
\psline(0,0)(2,0)
\psline(4,0)(8,0)
\psdots(0,0)(8,0)$
\endpspicture

\noindent sending $x_0$ to $y_0$, $x_1$ to $y_1$, the top edge to the top edge, and the bottom edge to the bottom edge; however the last two maps may be different.  Note that the maps are only required to be continuous so may still be quite pathological; one consequence of using Trimble's operad will be that we eliminate some of this pathology.
\end{myexample}

\begin{myexample}
For the 1-cells of $K^{(n)}$ the arities are just natural numbers $k$, and the geometric realisation of the pasting diagram of arity $k$ is just the closed interval $[k]$.  Thus the elements of $K^{(n)}$ of arity $k$ are the same as the points of $E(k)$, for Trimble's operad $E$.  This indicates the sense in which $K$ is a higher-dimensional version of $E$.
\end{myexample}

\begin{myremark}
Since each $|\alpha|$ is contractible it follows that $K$ is contractible.  Furthermore, given a space $X$ we can associate to it a globular set whose 0-cells are the points of the space, 1-cells the paths, 2-cells homotopies between paths, 3-cells the homotopies between homotopies, and so on; $K$ then acts naturally on this globular set making it an $n$-category, which Batanin defines to be the fundamental $n$-groupoid of $X$.
\end{myremark}

\subsection{Comparison}

We now compare the globular operad associated to Trimble's operad $E$ with Batanin's operad $K$.  The main idea of this construction is that we need to turn Trimble's maps of intervals into Batanin's maps of discs; we do this via the following topological suspension functor.

\begin{mydefinition}

We define a functor $\sigma: \cat{Top} \lra \cat{Top}$ as follows.  Given a space $X$, define
\[\sigma X = I \times X / \sim\]
where $\sim$ is the equivalence relation defined by $(0,x) \sim (0,x')$ and $(1,x) \sim (1,x')$ for all $x,x' \in X$.
\end{mydefinition}

Some examples are sketched below.
\[\begin{array}{|c|c|}
\hline \hspace{4em} & \hspace{9em} \\[-6pt]
X \in \cat{Top} &\sigma X \\[6pt]
\hline & \\[-6pt]
\bullet & \parbox[c]{4em}{\xy
(0,0)*{\bullet}="1";
(10,0,)*{\bullet}="2";
"1";"2" **\dir{-}; \endxy} \\[6pt]
\hline &\\
\parbox[c]{4em}{\xy
(0,0)*{\bullet}="1";
(0,10,)*{\bullet}="2";\endxy} &
\parbox[c][4em]{6em}{
\psset{yunit=0.6cm,xunit=0.6cm,runit=0.6cm}
\pspicture(-6.8,-2)(-6,2)
\psdots(-6,0)(-4,0)
\parabola(-6,0)(-5,0.5)
\parabola(-6,0)(-5,-0.5)
\endpspicture}
\\
\hline & \\
\parbox[c]{4em}{\xy
(0,0)*{\bullet}="1";
(0,10,)*{\bullet}="2";
"1";"2" **\dir{-}; \endxy} &
\parbox[c][4em]{6em}{
\psset{yunit=0.6cm,xunit=0.6cm,runit=0.6cm}
\pspicture(-6.8,-2)(-6,2)
\psdots(-6,0)(-4,0)
\parabola[fillstyle=hlines,hatchsep=2pt](-6,0)(-5,0.5)
\parabola[fillstyle=hlines,hatchsep=2pt](-6,0)(-5,-0.5)
\endpspicture}
\\[2em]
\hline
\end{array}\]

\begin{myremarks} \hspace*{1em}
\begin{enumerate}
\item As the above sketch suggests, iterating $\sigma$ gives us a way of constructing topological $m$-discs (up to homeomorphism), as for each $m \geq 0$ we have
    \[\sigma D^m \iso D^{m+1}.\]
\item Topological suspension is related to the globular set suspension functor
    \[\sigma: \cat{GSet} \lra \cat{GSet}\]
mentioned at the beginning of the section. It is not too hard to prove that for a pasting diagram $\alpha$ we have
\[|\sigma \alpha | \iso \sigma |\alpha|.\]

\item Suspension gives us a way of using induction to construct geometric realisation -- if $\alpha$ is an $m$-pasting diagram given by the series $(\alpha_1, \ldots , \alpha_k)$ of $(m-1)$-pasting diagrams, it follows from (2) that
    \[|\alpha| = \sigma |\alpha_1| +_0 \cdots +_0 \sigma |\alpha_k|\]
    where $\sigma X +_0 \sigma Y$ denotes the obvious pushout

    
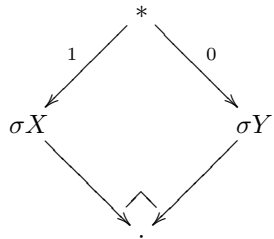
\[\xy
(0,15)*+{\ast}="t";
   (-15,0)*+{\sigma X}="l";
   (15,0)*+{\sigma Y}="r";
   (0,-15)*+{\cdot}="b";
       {\ar_{1} "t";"l"};
       {\ar^{0} "t";"r"};
       {\ar "r";"b"};
       {\ar "l";"b"};
(0,-10)*{\ucorner};
\endxy \]
\end{enumerate}
\end{myremarks}

We are now ready to define our comparison morphism.

\begin{theorem}
Let $n \geq 0$. Let $Q^{(n)}$ be the globular operad associated to Trimble's original theory of $n$-categories, and $K^{(n)}$ the operad defined above.  Then there is a morphism of operads
\[\theta: Q^{(n)} \lra K^{(n)}.\]
\end{theorem}

\begin{proof}
Let $\alpha$ be an $m$-stage tree as below as usual:
%
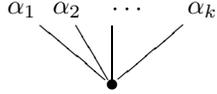
\[\xy
(-12,10)*+{\alpha_1}="2";
(-6,10)*+{\alpha_2}="3";
(0,10)*+{\phantom{\alpha_3}}="4";
(2,10)*+{\cdots};
(12,10)*+{\alpha_k}="5";
(0,0)*{\bullet}="1";
"1";"2" **\dir{-};
"1";"3" **\dir{-};
"1";"4" **\dir{-};
"1";"5" **\dir{-};
\endxy.\]
We know from Proposition~\ref{propositionfourpointfive} that
\[Q^{(n)}_\alpha = P_{n-1}(k)_{m-1} \times Q^{(n-1)}_{\alpha_1} \times \cdots \times Q^{(n-1)}_{\alpha_k}\]
and we use this formula to construct our morphism by induction over $n$. For $n=0$ we have $Q^{(0)}=K^{(0)}$, so we set $\theta$ to be the identity.

For $n \geq 1$ consider $\tau \in Q^{(n)}_\alpha$ given by
\[(\beta, \tau_1, \cdots, \tau_k) \in P_{n-1}(k)_{m-1} \times Q^{(n-1)}_{\alpha_1} \times \cdots \times Q^{(n-1)}_{\alpha_k}.\]
We need to construct a map $\theta\tau : D^m \lra |\alpha|$ preserving the boundary.  By induction we already have for each $1 \leq i \leq k$ a boundary-preserving map
\[\theta \tau_i : D^{m-1} \lra |\alpha_i| \]
and thus
\[\sigma \theta \tau_i : \sigma D^{m-1} \iso D^m \lra \sigma |\alpha_i|.\]
We form the following composite
\[
\xy
(0,0)*+{D^m}="1";
(30,-3)*+{\underbrace{{\white D^m +_0 \cdots +_0 D^m}}_{\textsf{$k$ times}}}="";
(30,0)*+{D^m +_0 \cdots +_0 D^m}="2";
(100,0)*+{\sigma |\alpha_1| +_0 \cdots +_0 \sigma |\alpha_k| \iso |\alpha|}="3";
{\ar^{} "1";"2"};
{\ar^>>>>>>>>>>>>>>>>>{\sigma \theta\tau_1 +_0 \cdots +_0 \sigma \theta \tau_k} "2";"3"};
\endxy
\]
where the first component is given as follows.  We have a map
\[
\xy
(0,0)*+{I^m}="1";
(45,0)*+{I^{m-1} \times [k]}="2";
(95,-3)*+{\underbrace{{\white D^m +_0 \cdots +_0 B^m}}_{\mbox{\textsf{$k$ times}}}}="3";
(95,0)*+{D^m +_0 \cdots +_0 D^m}="3";
(67,5)*+{\mbox{\textsf{quotient}}}="";
{\ar^>>>>>>>>>>>>>>>>>>{\overbrace{{\scriptstyle 1 \times \cdots \times 1}}^{\mbox{\textsf{$m-1$ times}}} \hspace{-1.3em} \times  \beta} "1";"2"};
{\ar^>>>>>>>>>>>>>>{\sim} "2";"3"};
(0,-10)*+{(x_1, \ldots, x_m)}="11";
(45,-10)*+{\big(x_1, \ldots, x_{m-1}, \beta(x_1, \ldots, x_m)\big)}="12";
{\ar@{|->} "11";"12"};
\endxy
\]
and it is easy to check that this respects the equivalence relation giving $D^m = I^m/\sim$ on each component, hence induces a map
\[D^m \lra D^m +_0 \cdots +_0 D^m.\]
It is straightforward to check that this is an operad morphism, essentially by induction and continuity of $\beta$.

\end{proof}

In effect where Batanin's operations give us \emph{all} reparametrisations of the topological pasting diagram, Trimble's operations just give us those reparametrisations which come from repeated suspensions of 1-dimensional reparametrisations.  This eliminates many of the pathological reparametrisation maps allowed by Batanin's operad.
In future work we hope to prove that $\theta$ is an embedding, and to use $Q^{(n)}$ and other convenient suboperads of $K$ to study the modelling of homotopy types.  Batanin has conjectured that his $n$-groupoids model $n$-types.  Cisinski has proved in \cite{cis1} that the \emph{$\omega$-categories} model homotopy types, and for the case of $\omega$-groupoids has demonstrated a faithful and conservative embedding of the homotopy category of spaces in the homotopy category of $\omega$-groupoids.  However, the full result for higher groupoids remains an open problem.

%
%


\ed